\documentclass[10pt]{article}
\usepackage{amssymb}
\usepackage{latexsym}
\usepackage{graphicx}
\usepackage{amsmath}
\usepackage{hyperref}

\newenvironment{proof}{{\em Proof.} }{\hfill$\Box$\vspace{0.1in}}

\newtheorem{proposition}{Proposition}[section]
\newtheorem{lemma}{Lemma}[section]
\newtheorem{theorem}{Theorem}[section]

\numberwithin{equation}{section}

\newcommand{\Z}{{I\!\! Z}}
\newcommand{\N}{{I\!\! N}}
\newcommand{\R}{\mathbb{ R}}
\newcommand{\C}{\mbox{\rm C\hspace{-0,4em}\rule{0,05ex}{1,55ex}\hspace{0,1em}}}

\title{On the well-posedness   of   the Cauchy problem for the generalized Korteweg-de Vries-Burgers
equation \thanks{Supported by NSFC (10571158) and Zhejiang
Provincial NSF of China (Y605076)}}
\author{ Ruying  Xue\\ {\small Department of  Mathematics, Zhejiang University,}\\
{\small Hangzhou 310027, Zhejiang, P.  R.  China }
\\{\small e-mail address: ryxue@zju.edu.cn} }
\date{}
 
\begin{document}
\maketitle
\pagestyle{plain}
\begin{quote}\small \bf Abstract\quad \rm Considered is  the generalized Korteweg-de  Vries-Burgers equation
$$ u_{t}+u_{xxx}+uu_{x}+|D_{x}|^{2\alpha }u=0,\quad t\in \mathbb{R}^{+},\,
x\in \mathbb{R}, $$
 with $0\leq \alpha\le 1$. We prove  a  sharp  results on the associated Cauchy problem  in the
 Sobolev  space $ {H}^s(\mathbb{R})$.  For $s>-\min\{\frac {3+2\alpha}4, 1\}$ we give the well-posedness
 of solutions of  the Cauchy problem,
while for $\frac 12\le\alpha\le 1$ and for  $s<-\min\{\frac
{3+2\alpha}4, 1\}$ we show some ill-posedness issues.


\noindent\bf 2000 MR Subject Classification\quad {\rm 35Q53,
35Q60} \end{quote}\normalsize

\setcounter{section}{0}
\section{Introduction and statement of the result}

In this paper we consider the cauchy problem associated with the
generalized Korteweg-de Vries-Burgers equation
\begin{equation}\left\{\begin{array}{l}u_{t}+u_{xxx}+uu_{x}+|D_{x}|^{2\alpha }u=0,\quad t\in \R^{+},\, x\in \R\\
u(0)=\varphi (x),\end{array}\right.\label{1}\end{equation} where,
$0\le\alpha\le 1$,  $|D_{x}|^{2\alpha }$ is   the Fourier
multiplier associated with  the symbol $|\xi|^{2\alpha }$.

Equation (\ref{1}) has been derived as a model for the propagation
of weakly nonlinear dispersive long waves in some physical
contexts when dissipative effects occur (see \cite{Ott}). The long
time asymptotic behavior of its solutions has been studied in
numerous papers (see \cite{Bona} and references therein ).

When $\alpha=0$,  (\ref{1}) is the Korteweg-de Vries equation. The
best known results on the Cauchy problem  for the Korteweg-de
Vries equation have been derived by Kenig, Ponce and Vega (see
\cite{Kenig1}, \cite{Kenig2}). They proved that the Cauchy problem
for the KdV equation is locally well-posed in $H^s(\R)$ for
$s>-\frac 34$, and that the flow-map for the KdV equation is  not
locally uniformly continuous in $H^s(\R)$ for $s<-\frac 34$. For
the Cauchy problem of  the dissipative Burgers equation
$$u_t-u_{xx}+uu_x=0,$$
it is known that the local well-posedness in $H^s(\R)$ holds for
$s\ge -\frac 12$ (see \cite{Bek}), and  some non-uniqueness
phenomena occur for $s<-\frac 12$( see \cite{Dix}). When
$\alpha=1$, (\ref{1}) is the Korteweg-de Vries-Burgers equation.
 Molinet and Ribaud in \cite{Monlinet} proved that the Korteweg-de Vries-Burgers
equation is globally well-posed in $H^s(\R)$ for $s>-1$ and
ill-posed in $H^s(\R)$ for $s<-1$. They  proved that the Cauchy
problem (\ref{1}) associated  with $0\le \alpha\le 1$ is ill-posed
in the homogenous Sobolev space $ \dot{H}^s(\R)$ for $s<\frac
{\alpha-3}{2(2-\alpha)}$, and conjectured that the well-posedness
in ${H}^s(\R)$ for $s>\frac {\alpha-3}{2(2-\alpha)}$ could be
proved. The aim of this paper is to answer this open problem. We
prove that (\ref{1})  is well-posed in  the Sobolev space
${H}^s(\R)$ for $s>-\min\{\frac {3+2\alpha}4, 1\}$. Note that
$-\min\{\frac {3+2\alpha}4, 1\}< \frac {\alpha-3}{2(2-\alpha)}$
for $0<\alpha<1$.

 Let $<\cdot>=(1+|\cdot|
^{2})^{\frac{1}{2}}$. We define
$$ {X}_{\alpha }^{b,s}=\{u\in {\cal S}'(\R^2):
\, \parallel u\parallel_{ {X}_{\alpha}^{b,s}}<+\infty\},$$
$$ {X}_{\alpha,T }^{b,s}=\{u: \, \exists v\in  {X}_{\alpha
}^{b,s}\, \mbox{ satisfying }\, u=v \, \mbox{in}\, \R\times [0,
T]\},$$ with
$$\parallel u\parallel _{ {X}_{\alpha}^{b,s}}
=\|<i(\tau -\xi^{3})+|\xi|^{2\alpha }>^{b}<\xi>^{s}\hat
u(\xi,\tau)\|_{L^2(\R^2)},$$
$$\parallel u\parallel _{ {X}_{\alpha,T}^{b,s}}
=\inf\{\parallel v\parallel _{ {X}_{\alpha}^{b,s}}: \, v\in
 {X}_{\alpha }^{b,s}\, \mbox{satisfying }\, u=v\,  \mbox{in }
\,  \R\times [0, T]\,\}.$$ Let $ {H}^{s}(\R)$  be the usual
 Sobolev space.  Our main result is
\begin{theorem}\label{T1.1} Let $\varphi \in  {H}^{s}(\R)$ with
$s>-\min\{\frac {3+2\alpha}4, 1\}$. For any  $T>0$, there exists a
unique solution $u$ of (\ref{1}) satisfying
$$u\in
Z_{T}=C([0,T], {H}^{s}(\R))\cap  {X}_{\alpha
,T}^{\frac{1}{2},s}.$$
 Moreover the map $\varphi \mapsto u$ is smooth from
$ {H}^{s}(\R)$ to $Z_{T}$ and $u$ belongs to $C((0, +\infty),
H^\infty (\R))$.
\end{theorem}

{\bf Remark }\quad  For $s<\frac{\alpha -3}{2(2-\alpha )}$,
Molinet and Ribaud (see Remark 1 and Theorem 2 in \cite{Monlinet}
) proved that the flow-map
$$\varphi\mapsto u(t), \, t\in [0,T]$$
is not $C^2$ differentiable at zero from the homogenous Sobolev
$\dot{H}^s(\R)$ to $C([0, T]; \dot{H}^s(\R))$.

The result is optimal in the case $\frac 12\le\alpha\le 1$.
\begin{theorem}\label{T2} Let $\frac 12\le\alpha\le 1$ and $s<-1$.
Then there does not exists  $T>0$ such that the Cauchy problem
(\ref{1}) has a unique local solution $u$ defined on the interval
$[0, T]$, and such that the flow-map
$$\varphi\to u(t), t\in [0,T]$$
is $C^2$ differentiable at zero from $ H^s(\R)$ to  $C([0, T],
H^s(\R))$. \end{theorem}

 In this paper, we use $A\lesssim B$ to denote
the statement that $A\le CB$ for some large constant C which may
vary from line to line, and similarly use $A\ll B$ to denote the
statement $A\le C^{-1}B$. We use $A\sim B$ to denote the statement
that $A\lesssim B\lesssim A$. Any summations over capitalized
variables such as $N_j$, $L_j$ , $H$ are presumed to be dyadic,
i.e. these variables range over numbers of the form $2^k$ for
$k\in \Z$ or for $k\in \N$ . In addition to the usual notation
$\chi_E$ for characteristic functions, we define $\chi_P$ for
statements $P$ to be $1$ if $P$ is true and 0 otherwise, e.g.
$\chi_{1\le |\xi|\le 2}$.

We adopt the following summation conventions. Any summation of the
form $L_{max}\sim\cdot$ is a sum over the three dyadic variables
$L_1, L_2, L_3\gtrsim 1$, thus for instance
$$\sum_{L_{max}\sim H}
:= \sum_{L_1,L_2,L_3\gtrsim 1; L_{max}\sim H}.$$
 Similarly, any
summation of the form $N_{max}\sim\cdot$ sum over the three dyadic
variables $N_1, N_2, N_3> 0$, thus for instance $$\sum_{
N_{max}\sim N_{med}\sim N} :=\sum_{N_1 ,N_2, N_3>0; N_{max}\sim
N_{med}\sim N}.$$

The rest of this paper is organized as follows. In section 2 we
give some linear estimates. In section 3 we prove the crucial
bilinear estimates and give the proof of Theorem \ref{T1.1}. The
ill-posedness is given in section 4.

\section{Linear estimates}

Let $U(\cdot )$ be the free evolution of the KdV equation defined
by $U(t)=e^{itP(D_{x})}$, where $P(D_{x})$ is the Fourier
multiplier with the symbol $P(\xi )=\xi ^{3}$. Obviously $U(\cdot
)$ is a unitary group in $ {H}^{s}(\R), s\in \R$.  Since
$\mathcal{F}(U(-t)u)(\tau ,\xi )=\mathcal{F}(u)(\tau +\xi ^{3},\xi
)$,  one can rewrite the norm of $ {X}_{\alpha }^{b,s}$ as
$$\parallel u\parallel _{ {X}_{\alpha }^{b,s}}
=\parallel <i\tau +|\xi|^{2\alpha }>^{b}<\xi
>^{s}\mathcal{F}(U(-t)u)(\tau ,\xi )\parallel _{L^{2}(\R^{2})}.$$
Let $W(\cdot )$  be the semigroup associated with the free
evolution of (\ref{1}) defined by
\[ \mathcal{F}_{x}(W(t)\varphi )(\xi )=e^{it\xi ^{3}-t\mid
\xi \mid ^{2\alpha }}\hat{\varphi }(\xi ),\quad \varphi \in
\mathcal{S}'(\R),\, t\ge 0,\] and we extend $W(\cdot )$ to a
linear operator defined on the whole real axis by setting
\[ \mathcal{F}_{x}(W(t)\varphi )(\xi )=e^{it\xi ^{3}-|t| |\xi|^{2\alpha }}\hat{\varphi }(\xi ),\quad \varphi \in
\mathcal{S}'(\R),\, t\in\R.\]
 Let $\psi $ be a time cut-off function defined by
\[\psi \in C_{0}^{\infty }(\R),\, supp \psi \subset [-2,2],\, \psi \equiv 1 \, \mbox{ on } \,
[-1,1]\] and let $\psi _{T}(\cdot )=\psi (\cdot /T)$ for a given
$T>0$.

\begin{proposition}\label{L2.1} For $s\in \R$,  we have
\[\parallel \psi (t)W(t)\varphi \parallel_{ {X}_{\alpha
}^{\frac{1}{2},s}}\lesssim \parallel \varphi \parallel _{
{H}^{s}},\quad \forall \varphi \in
 {H}^{s}(\R).\]\end{proposition}
\begin{proof} Set $g_{\xi }=\psi (t)e^{-\mid
t\mid \mid \xi \mid ^{2\alpha }}$. For $b\in \{0,\,\frac{1}{2}\}$
we have
\begin{eqnarray*}\parallel g_{\xi }\parallel _{H_{t}^{b}}\leq \parallel <\tau >^{b}\hat{\psi }\parallel _{L^{1}}
\parallel
e^{-|t| \mid \xi \mid ^{2\alpha }}\parallel _{L^{2}}+\parallel
\hat{\psi }\parallel _{L^{1}}\parallel e^{-|t| \mid \xi \mid
^{2\alpha }}\parallel _{  {H_{t}^{b}}}.\end{eqnarray*} Since $\psi
\in C_{0}^{\infty }(R),supp\psi \subset [-2,2]$, we get $\parallel
<\tau
>^{b}\hat{\psi }\parallel _{L^{1}}\leq C$. Note that
\[\parallel e^{-\mid t\mid \mid \xi \mid ^{2\alpha }}\parallel
_{  {H_t^{b}}}\sim (\mid \xi \mid ^{2\alpha
})^{b-\frac{1}{2}}\parallel e^{-\mid t\mid }\parallel _{
{H_t^{b}}}.\] We deduce for $\mid \xi \mid \ge 1$
\begin{equation}\parallel g_{\xi }\parallel _{H_t^{b}}\lesssim (\mid \xi \mid ^{-\alpha }
+\mid \xi \mid ^{2\alpha b-\alpha }) \leq C\mid \xi \mid ^{2\alpha
(b-\frac{1}{2})},\label{1.4}\end{equation}
 and for $\mid \xi \mid \leq 1$,
\begin{equation}\parallel g_{\xi }\parallel
_{H_{t}^{b}}\leq \sum _{n=0}^{\infty }\frac{\mid \xi \mid
^{2\alpha n}}{n!}\parallel \psi (t)t^{n}\parallel _{H_{t}^{b}}\leq
\sum _{n=0}^{\infty }\frac{\mid \xi \mid ^{2\alpha
n}}{n!}\parallel \psi (t)t^{n}\parallel _{H_{t}^{1}}\lesssim 1.
\label{1.5}\end{equation}  A combination of (\ref{1.4}) with
(\ref{1.5}) yields
\begin{equation}\parallel g_{\xi }\parallel _{H_t^{b}}\lesssim <\xi
>^{\alpha (2b-1)}, \quad b=0\,\mbox{ or } \, \frac{1}{2}
.\label{2.2a}\end{equation} By  (\ref{2.2a}), we have
\begin{eqnarray*}&{}&\parallel \psi (t)W(t)\varphi \parallel_{ {X}_{\alpha }^{\frac{1}{2},s}}\\
&\lesssim& \left\| <\xi>^{s}\hat{\varphi }(\xi )\parallel <\tau
>^{\frac{1}{2}}\mathcal{F}_{t}(\psi (t)e^{-|t| \mid \xi \mid ^{2\alpha
}})(\tau )\parallel _{L_{\tau }^{2}}\right\|_{L_{\xi }^{2}}\\
&{}&+\left\| <\xi>^{s+\alpha }\hat{\varphi }(\xi )\parallel \psi
(t)e^{-|t| \mid \xi \mid ^{2\alpha }}\parallel
_{L_{t}^{2}}\right\|_{L_{\xi
}^{2}}\\
&\lesssim& \left\| <\xi>^{s}\hat{\varphi }(\xi )\parallel g_{\xi
}(t)\parallel _{H_{t}^{\frac{1}{2}}}\right\|_{L_{\xi }^{2}}
+\left\| <\xi>^{s+\alpha }\hat{\varphi }(\xi )\parallel g_{\xi
}(t)\parallel_{H_{t}^{0}}\right\|_{L_{\xi
}^{2}}\\
&\lesssim& \parallel <\xi>^{s}\hat{\varphi }(\xi )\parallel
_{L_{\xi }^{2}} +C\parallel <\xi>^{s}\hat{\varphi }(\xi )\parallel
_{L_{\xi }^{2}}\lesssim\parallel \varphi \parallel _{
{H}^{s}}.\end{eqnarray*}\end{proof}

 The following
proposition comes from Proposition 2 in \cite{Monlinet} (we
replace $\xi$ by $|\xi|^{2\alpha}$).

\begin{proposition}\label{L2.2} For $\omega \in \mathcal{S}(\R^{2})$ we define  $K_{\xi
}$ by
\[K_{\xi }(t)=\psi (t)\int _{R}\frac{e^{it\tau }-e^{-|t| \mid
\xi \mid ^{2\alpha }}} {i\tau +\mid \xi \mid ^{2\alpha
}}\hat{\omega }(\tau )d\tau .\] Then for all $\xi \in \R$,
\begin{equation}\left\|<i\tau +\mid \xi \mid ^{2\alpha
}>^{\frac{1}{2}}\mathcal{F}_{t}(K_{\xi })\right\|_{L^{2}(\R)}^{2}
\lesssim\left [\left(\int _{\R}\frac{\mid \hat{\omega }(\tau )\mid
}{<i\tau +\mid \xi \mid ^{2\alpha }>}d\tau \right )^{2} +\int
_{\R}\frac{\mid \hat{\omega }(\tau )\mid ^{2}}{<i\tau +\mid \xi
\mid ^{2\alpha }>}d\tau \right ].\label{1.6}
\end{equation}\end{proposition}

\begin{proposition}\label{2.3} For $s\in \R$ we have

$[a].$\,  for all $v\in {\cal S}(\R^{2})$,
\begin{eqnarray}&{}&\left\|\chi _{R^{+}}(t)\psi (t)\int
_{0}^{t}W(t-t')v(t')dt'\right\|_{ {X}_{\alpha
}^{\frac{1}{2},s}}\nonumber \\
&\lesssim &\| v\|_{ {X}_{\alpha }^{-\frac{1}{2},s}}+\left
(\int_{\R} <\xi>^{2s}(\int_{\R} \frac{\mid \hat{v}(\tau )\mid
}{<i\tau +\mid \xi \mid ^{2\alpha }>}d\tau )^{2}d\xi\right
)^{\frac{1}{2}};\label{0.7}\end{eqnarray}

$[b].$\,  for $0<\delta <\frac{1}{2}$ and   for all $v\in
 {X}_{\alpha }^{-\frac{1}{2}+\delta ,s}$,
\begin{equation}\left\|\chi _{R^{+}}(t)\psi (t)\int
_{0}^{t}W(t-t')v(t')dt'\right\|_{ {X}_{\alpha }^{\frac{1}{2},s}}
\lesssim \parallel v\parallel _{ {X}_{\alpha
}^{-\frac{1}{2}+\delta
,s}}.\label{0.8}\end{equation}\end{proposition}
\begin{proof}Assume that $v\in {\cal S}(\R^{2})$. Taking that for $x$-Fourier
transform we get
\begin{eqnarray*}&{}&\chi _{R^{+}}(t)\psi (t)\int _{0}^{t}W(t-t')v(t')dt'\\
&=& U(t)\chi _{R^{+}}(t)\psi (t)\int_{\R} e^{ix\xi }\int
_{0}^{t}e^{-(t-t')\mid \xi \mid ^{2\alpha
}}\mathcal{F}_{x}(U(-t')v(t'))dt'd\xi\\
&=& U(t)\chi _{R^{+}}(t)\psi (t)\int_{\R^2} e^{ix\xi }\hat{\omega
}(\tau ,\xi )e^{-t\mid \xi \mid ^{2\alpha }}\int _{0}^{t}e^{t'\mid
\xi \mid ^{2\alpha }}e^{it'\tau }
dt'd\xi d\tau \\
&=& U(t)\chi _{R^{+}}(t)\psi (t)\int_{\R^2}  e^{ix\xi }\hat{\omega
}(\tau ,\xi )\frac{e^{it\tau }-e^{-t\mid \xi \mid ^{2\alpha
}}}{i\tau +\mid \xi \mid ^{2\alpha }}d\xi d\tau \\
&=& U(t)\chi _{R^{+}}(t) \int_{\R} e^{ix\xi
}K_{\xi}(t)d\xi,\end{eqnarray*} where we denote by  $\omega
(t')=U(-t')v(t')$.  By Proposition \ref{L2.2}, we deduce
\begin{eqnarray*}&{}&\left\| \chi _{R^{+}}(t)\psi (t)\int
_{0}^{t}W(t-t')v(t')dt'\right\|_{ {X}_{\alpha
}^{\frac{1}{2},s}}\le \left\|<i\tau+|\xi|^{2\alpha}>^{\frac
12}<\xi>^s
{\cal F}_t(K_\xi (t))\right\|_{L^2(\R^2)}\\
&\lesssim & \left (\int_{\R} <\xi>^{2s}(\int _{\R}\frac{\mid
\hat{\omega }(\tau )\mid }{<i\tau +\mid \xi \mid ^{2\alpha
}>}d\tau )^{2}d\xi \right )^{\frac{1}{2}} +\left (\int_{\R}
<\xi>^{2s}\int _{\R}\frac{\mid \hat{\omega }(\tau )\mid
^{2}}{<i\tau +\mid \xi \mid ^{2\alpha }>}d\tau d\xi \right )^{\frac{1}{2}}\\
&\lesssim& \left (\int_{\R} <\xi>^{2s}(\int _{\R}\frac{\mid
\hat{v}(\tau )\mid }{<i\tau +\mid \xi \mid ^{2\alpha }>}\parallel
e^{-it\xi ^{3}}\parallel _{L^{\infty }}d\tau )^{2}d\xi
\right)^{\frac{1}{2}} +\parallel
v\parallel _{ {X}_{\alpha }^{-\frac{1}{2},s}}\\
&\lesssim& \parallel v\parallel _{ {X}_{\alpha
}^{-\frac{1}{2},s}}+\left (\int_{\R} <\xi>^{2s}(\int
_{\R}\frac{\mid \hat{v}(\tau )\mid }{<i\tau +\mid \xi \mid
^{2\alpha }>}d\tau )^{2}d\xi \right
)^{\frac{1}{2}}.\end{eqnarray*}
 We complete the
proof of (\ref{0.7}). Now we prove (\ref{0.8}). For
$\delta\in(0,\frac 12)$, obviously
$$\parallel v\parallel _{ {X}_{\alpha
}^{-\frac{1}{2},s}}\leq
\parallel v\parallel _{ {X}_{\alpha }^{-\frac{1}{2}+\delta ,s}}.$$
 By $H\ddot
older$ inequality, we have
$$\int_{\R} \frac{\mid \hat{v}(\tau )\mid
}{<i\tau +\mid \xi \mid ^{2\alpha }>}d\tau \lesssim  \left\| |
\hat{v}(\tau )| <i\tau +\mid \xi \mid ^{2\alpha
}>^{-\frac{1}{2}+\delta }\right\|_{L^{2}(_{\R})},$$  and then
$$\left (\int_{\R} <\xi>^{2s} (\int_{\R}
\frac{\mid \hat{v}(\tau )\mid }{<i\tau +\mid \xi \mid ^{2\alpha
}>}d\tau )^{2}d\xi\right )^{\frac{1}{2}}\lesssim \parallel
v\parallel _{ {X}_{\alpha }^{-\frac{1}{2}+\delta ,s}}.$$
\end{proof}

\begin{proposition}\label{L2.4} Let $s\in \R$, $\delta >0$. For all $f\in  {X}_{\alpha }^{-\frac{1}{2}
+\delta ,s}$, one has
\begin{equation}t\mapsto \int _{0}^{t}W(t-t')f(t')dt'\in C(\R^{+}, {H}^{s+2\delta }). \label{9.1}\end{equation}
Moreover, if $\{f_{n}\}$ is a sequence with $f_{n}\rightarrow 0$
in $ {X}_{\alpha }^{-\frac{1}{2}+\delta ,s}$ as $n\rightarrow
\infty $, then
\begin{equation}\left\| \int _{0}^{t}W(t-t')f_{n}(t')dt'\right\|_{L^{\infty }(\R^{+}, {H}^{s+2\delta })}
\rightarrow 0,\quad  n\rightarrow\infty.
 \label{10}\end{equation}\end{proposition}
\begin{proof} The proof is similar to that of  Proposition 4 in \cite{Monlinet}, we omit it.\end{proof}

\section{A bilinear estimate and the proof of Theorem \ref{T1.1}}

Let $Z$ be any abelian additive group with an invariant measure
$d\eta$. For any integer $k\ge 2$, we   denote by $\Gamma_k(Z)$
the hyperplane
$$\Gamma_k(Z)=\{(\eta_1,\cdots, \eta_k)\in Z^k: \eta_1+\cdots
+\eta_k=0\},$$  we endow with the obvious measure
$$\int_{\Gamma_k(Z)}f:=\int_{Z^{k-1}}f(\eta_1,\cdots,\eta_{k-1},-\eta_1-\cdots-\eta_{k-1})d\eta_1\cdots
d\eta_{k-1}.$$ We define a $[k; Z]$-multiplier to be any function
$m: \Gamma_k(Z)\to \C$ . If m is a $[k; Z]$-multiplier, we define
$\|m\|_{[k; Z]}$ to be the best constant such that the inequality
$$\left |\int_{\Gamma_k(Z)}m(\eta)\Pi_{j=1}^kf_j(\eta_j)\right |\le
\|m\|_{[k;Z]}\Pi_{j=1}^k\|f_j\|_{L^2(Z)},$$ holds for all test
functions $f_j$ on $Z$.

In the sequel,  we choose $Z=\R\times \R$, $k=3$ and $\eta=(\tau,
\xi)$. For $N_1, N_2, N_3 > 0$, we define the quantities $
N_{max}\ge N_{med}\ge N_{min}$ to be the maximum, median and
minimum of $N_1, N_2, N_3$ respectively. Similarly define $
L_{max}\ge L_{med}\ge L_{min}$ whenever $L_1, L_2, L_3\ge 1$.
Define
$$h_j(\xi_j)=i\xi_j^3-|\xi_j|^{2\alpha}, \lambda_j=i\tau_j-h_j(\xi_j),\,  j=1,2,3,$$
and
$$h(\xi)=h_1(\xi_1)+h_2(\xi_2)+h_3(\xi_3).$$
We  shall take homogenous dyadic decomposition of the variable
$|\xi_j|\sim N_j>0$, and take non-homogenous dyadic decomposition
of the variable $|\lambda_j|\sim L_j\ge 1$ as well as the function
$|h(\xi)|\sim H\ge 1$ ( here the notations $|\lambda_j|\sim 1$ and
$|h(\xi)|\sim 1$ mean $|\lambda_j|\le 1$, $|h(\xi)|\le 1$,
respectively ). Define
$$X_{N_1,N_2,N_3;H;L_1,L_2,L_3} :=\chi_{|h(\xi)|\sim
H}\Pi_{j=1}^3\chi_{|\xi_j|\sim N_j} \chi_{|\lambda_j|\sim L_j}.$$
\begin{lemma}\label{L3.1} Let $N_1, N_2, N_3>0$, $L_1, L_2, L_3\gtrsim
1$ and $H\gtrsim 1$ satisfy
\begin{equation}N_{max}\sim N_{med}, L_{max}\sim \max\{H,L_{med}\}, H\sim
\max\{N_{max}^2N_{min},N_{max}^{2\alpha}\}.
\label{3.1a}\end{equation}

(1).  In the high modulation case $L_{max}\sim L_{med}\gg H$ we
have
\begin{equation}\left\|X_{N_1,N_2,N_3;H;L_1,L_2,L_3}\right\|_{[3,\R\times
\R]}\lesssim L_{min}^{\frac 12}N_{min}^{\frac
12}.\label{9}\end{equation}

(2). In the low modulation case $L_{max}\sim H$,

{}\qquad (2a). if $N_{max}\sim N_{med}\sim N_{min}$,  we have
\begin{equation}\left\|X_{N_1,N_2,N_3;H;L_1,L_2,L_3}\right\|_{[3, \R\times
\R]}\lesssim L_{min}^{\frac 12}\min\{N_{max}^{-\frac
14}L_{med}^{\frac 14}, L_{med}^{\frac
1{4\alpha}}\};\label{12}\end{equation}

{}\qquad (2b). if $N_2\sim N_3\gg N_1$ and $H\sim L_1\geq
L_2,L_3$,
 we have, for any $\beta\in (0,2]$,
\begin{equation}\left\|X_{N_1,N_2,N_3;H;L_1,L_2,L_3}\right\|_{[3, \R\times
\R]}\lesssim L_{min}^{\frac 12}\min\{N_1^{\frac 12},
L_{med}^{\frac 1{4\alpha}}, N_2^{\frac{\beta-2}{2\beta}}
N_1^{-\frac{1}{2\beta}}L_{med}^{\frac{1}{2\beta}}\};\label{13}\end{equation}

{}\qquad Similarly for permutations;

{}\qquad (2c).  In all other cases, we have
\begin{equation}\left\|X_{N_1,N_2,N_3;H;L_1,L_2,L_3}\right\|_{[3,\R\times
\R]}\lesssim L_{min}^{\frac 12}\min\{N_{max}^{-1}L_{med}^{\frac
12}, L_{med}^{\frac 1{4\alpha}}, N_{min}^{\frac
12}\}.\label{14}\end{equation}\end{lemma}

\begin{proof}We consider the high modulation case $L_{max}\sim L_{med}\gg
H$.  By using  the comparison principle (Lemma 3.1 in \cite{Tao}),
we have (without loss of generality we assume $L_1\ge L_2\ge L_3$
and $N_1\ge N_2\ge N_3$)
\begin{equation}\left\|X_{N_1,N_2,N_3;H;L_1,L_2,L_3}\right\|_{[3,\R\times
\R]}\lesssim \left\|\chi_{|\lambda_3|\sim L_{3}}\chi_{|\xi_3|\sim
N_3}\right\|_{[3,\R\times \R]}\end{equation} By  Lemma 3.14 and
Lemma 3.6 in \cite{Tao},
\begin{equation}\left\|\chi_{|\lambda_3|\sim L_{3}}\chi_{|\xi_3|\sim
N_3}\right\|_{[3,\R\times \R]}\lesssim \left\|
\|\chi_{|\lambda_3|\sim L_{3}}\|_{[3, \R]}\chi_{|\xi_3|\sim
N_3}\right\|_{[3,\R]}\lesssim L_3^{\frac 12}N_3^{\frac
12}.\label{8}\end{equation} Although we derived (\ref{8}) assuming
$L_1\ge L_2\ge L_3$ and $N_1\ge N_2\ge N_3$, it is clear from
symmetry that
\begin{equation}\left\|X_{N_1,N_2,N_3;H;L_1,L_2,L_3}\right\|_{[3,\R\times
\R]}\lesssim L_{min}^{\frac 12}N_{min}^{\frac
12}.\label{9.2}\end{equation}

We now consider the low modulation case $H\sim L_{max}$. Suppose
for the moment that $N_1\ge N_2 \ge N_3$. The $\xi_3$ variable  is
currently localized to the annulus $\{|\xi_3|\thicksim N_{3}\}$.
By a finite partition of unity we can restrict it further to a
ball $\{|\xi_3-\xi_3^0|\ll  N_{3}\}$ for some $|\xi_3^0|\thicksim
N_{3}$. Then by Box Localization ( Lemma 3.13 in \cite{Tao} ) we
may localize $\xi_1, \xi_2$ similarly to regions
$\{|\xi_1-\xi_1^0|\ll  N_{3}\}$ and $\{|\xi_2-\xi_2^0|\ll N_{3}\}$
where $|\xi_j^0|\thicksim N_j$. We may assume that
$|\xi_1^0+\xi_2^0+\xi_3^0|\ll N_3$ since we have
$\xi_1+\xi_2+\xi_3=0$. We  summarize this symmetrically as
\begin{equation}\left\|X_{N_1,N_2,N_3;H;L_1,L_2,L_3}\right\|_{[3,\R\times
\R]}\lesssim \left\|\chi_{|h(\xi)|\thicksim
H}\Pi_{j=1}^3\chi_{|\lambda_j|\sim L_{j}}\chi_{|\xi_j-\xi_j^0|\ll
N_{min}}\right\|_{[3,\R\times \R]},\end{equation} for some
$\xi_1^0, \xi_2^0,  \xi_3^0$ satisfying
$$|\xi_j^0|\thicksim N_j, |\xi_1^0+\xi_2^0+\xi_3^0|\ll N_{min}.$$

Without loss of generality, we assume  $L_1\ge L_2 \ge L_3$. By
Lemma 3.6 , Lemma 3.1 and Corollary 3.10 in \cite{Tao} we get
\begin{eqnarray}&{}&\left\|X_{N_1,N_2,N_3;H;L_1,L_2,L_3}\right\|_{[3,\R\times
\R]}\nonumber\\
&\lesssim & \left\|\chi_{|h(\xi)|\thicksim H}\Pi_{j=2
}^3\chi_{|\xi_j-\xi_j^0|\ll  N_{min}}\chi_{|\lambda_j|\sim
L_{j}}\right\|_{[3,\R\times \R]}\nonumber\\
&\lesssim & \left |\left\{(\tau_2,\xi_2): |\xi_2-\xi_2^0|\ll
N_{min},
|i\tau_2-h_2(\xi_2)|\thicksim L_2,\right. \right.\nonumber\\
&{}&\qquad \quad\left. \left. |\xi-\xi_2-\xi_3^0|\ll N_{min},
|i(\tau-\tau_2)-h_3(\xi-\xi_2)|\thicksim L_3\right \}\right
|^{\frac 12}\end{eqnarray} for some $(\tau, \xi)\in \R\times \R$.
For fixed $\xi_2$, the set of possible $\tau_2$ ranges in an
interval of length $O(\min\{L_2, L_3\})$, and vanishes unless
$$|i\tau-h_2(\xi_2)-h_3(\xi-\xi_2)|= O(\max\{L_2,L_3\}).$$
Then we get, for  some $(\tau, \xi)\in \R\times \R$,
\begin{eqnarray*}&{}&\left\|X_{N_1,N_2,N_3;H;L_1,L_2,L_3}\right\|_{[3,\R\times
\R]} \lesssim  L_3^{\frac 12}|\{\xi_2: |\xi_2-\xi_2^0|\ll N_{min},\nonumber\\
&{}&\qquad\qquad \quad |\xi-\xi_2-\xi_3^0|\ll N_{min},
|i\tau-h_2(\xi_2)-h_3(\xi-\xi_2)|=O(L_2)\}|^{\frac
12}.\end{eqnarray*} Note that the inequality
$|\xi-\xi_2-\xi_3^0|\ll N_{min}$ implies $|\xi-\xi_1^0|\ll
 N_{min}$. Then we have
\begin{eqnarray}&{}&\left\|X_{N_1,N_2,N_3;H;L_1,L_2,L_3}\right\|_{[3,\R\times
\R]}\nonumber\\
&\lesssim& L_3^{\frac 12}|\{\xi_2: |\xi_2-\xi_2^0|\ll N_{min},
 |\xi-\xi_1^0|\ll N_{min},\nonumber\\
 &{}&\qquad \qquad
|i\tau-h_2(\xi_2)-h_3(\xi-\xi_2)|=O(L_2)\}|^{\frac
12}.\label{11}\end{eqnarray} To compute the right-hand side of the
expression (\ref{11}) we   use the identity
$$|i\tau-h_2(\xi_2)-h_3(\xi-\xi_2)|=\left |i\tau-3i\xi(\xi_2-\frac{\xi}2)^2+i\frac{\xi^3}4+(|\xi_2|^{2\alpha}
+|\xi_2-\xi|^{2\alpha})\right |=O(L_2),$$
 which implies
\begin{equation}3\xi(\xi_2-\frac{\xi}2)^2+\frac{\xi^3}4 =\tau+O(L_2)\label{15}\end{equation}
and
\begin{equation}|\xi_2|^{2\alpha}+|\xi_2-\xi|^{2\alpha}=O(L_2).\label{16}\end{equation}
We need only consider three cases: $N_1\sim N_2\sim N_3$, $N_1\sim
N_2\gg N_3$, and $N_2\sim N_3\gg N_1$. (The case $N_1\sim N_3\gg
N_2$ then follows by symmetry).

If $N_1\sim N_2\sim N_3$, by $|\xi-\xi_1^0|\ll N_{min}$  we deduce
$|\xi|\sim N_1$. we see from (\ref{15}) that $\xi_2$ variable is
contained in the union of two intervals of length $O(N_1^{\frac
12}L_2^{\frac 12})$ at worst, and from (\ref{16}) that $|\xi_2|\le
L_2^{\frac 1{2\alpha}}$, and (\ref{12}) follows from (\ref{11}).

If $N_1\sim N_2\gg N_3$,  by $|\xi-\xi_1^0|\ll N_{min}$, $
|\xi_2-\xi_2^0-\frac{\xi-\xi_1^0}2-\xi_3^0|\ll N_{min}$ and
$$\left |\xi_2-\frac{\xi}2\right |=\left |\xi_2-\xi_2^0-\frac{\xi-\xi_1^0}2-\xi_3^0-\frac{\xi_1^0}2\right |$$
we get $|\xi|\sim N_1$ and $|\xi_2-\frac{\xi}2|\sim N_1$. we see
from (\ref{15}) that $\xi_2$ variable is contained in the union of
two intervals of length $O(N_1^{-2}L_2)$ at worst, and from
(\ref{16}) that $|\xi_2|\le L_2^{\frac 1{2\alpha}}$, and
(\ref{14}) follows from (\ref{11}).

If $N_2\sim N_3\gg N_1$, then we must have $|\xi|\sim N_1$ and
$|\xi_2-\frac{\xi}2|\sim N_2$.
 For a given $\beta\in (0, 2]$, we have $|\xi ||\xi_2-\frac{\xi}2|^{2-\beta}\sim N_1N_2^{2-\beta}$.
we see from (\ref{15}) that $\xi_2$ variable is contained in the
union of two intervals of length $O(N_1^{-\frac
1{\beta}}N_2^{\frac{\beta-2}{\beta}}L_2^{\frac 1{\beta}})$ at
worst, and from (\ref{16}) that $|\xi_2|\le L_2^{\frac
1{2\alpha}}$. (\ref{13}) follows from (\ref{11}) and the fact that
$|\xi_2-\xi_2^0|\ll N_1$ for some $|\xi_2^0|\ll N_2$.\end{proof}

\begin{lemma}\label{L3.2} For a given $\rho\in (\frac 12,
\min\{\frac {3+2\alpha}4, 1\})$ and for any $\delta>0$ small we
have
\begin{equation}\left\|\frac{\xi_3 <\xi_1>^{\rho}<\xi_2>^{\rho}<\xi_3>^{-\rho} }{<\lambda_1>^{\frac
12}<\lambda_2>^{\frac 12}<\lambda_3>^{\frac
12-\delta}}\right\|_{[3, \R\times \R]}\lesssim
1.\label{0}\end{equation}\end{lemma}

\begin{proof} We have
\begin{eqnarray*}&{}&\left\|\frac{\xi_3
<\xi_1>^{\rho}<\xi_2>^{\rho}<\xi_3>^{-\rho}\chi_{|\xi_1|\lesssim
1}\chi_{|\xi_2|\lesssim 1}\chi_{|\xi_3|\lesssim
1}}{<\lambda_1>^{\frac 12}<\lambda_2>^{\frac 12}<\lambda_3>^{\frac
12-\delta}}\right\|_{[3,\R\times \R]}\nonumber\\
&\lesssim &\left\|\frac{\chi_{|\xi_1|\lesssim
1}\chi_{|\xi_2|\lesssim 1}\chi_{|\xi_3|\lesssim
1}}{<\lambda_1>^{\frac 12}<\lambda_2>^{\frac 12}<\lambda_3>^{\frac
12-\delta}}\right\|_{[3,\R\times \R]}.\end{eqnarray*} By taking
the non-homogenous dyadic decomposition of the variable
$|\lambda_j|\sim L_j\ge 1$( here the notation $|\lambda_j|\sim
L_j=1$ means $|\lambda_j|\le 1$), we get
\begin{eqnarray}&{}&\left\|\frac{\xi_3
<\xi_1>^{\rho}<\xi_2>^{\rho}<\xi_3>^{-\rho}\chi_{|\xi_1|\lesssim
1}\chi_{|\xi_2|\lesssim 1}\chi_{|\xi_3|\lesssim
1}}{<\lambda_1>^{\frac 12}<\lambda_2>^{\frac 12}<\lambda_3>^{\frac
12-\delta}}\right\|_{[3,\R\times \R]}\nonumber\\
&\lesssim &\sum_{L_1, L_2, L_3\gtrsim 1}
\frac{\left\|\Pi_{j=1}^3\chi_{|\xi_j|\lesssim
1}\chi_{|\lambda_j|\sim L_j}\right\|_{[3,\R\times
\R]}}{<L_1>^{\frac 12}<L_2>^{\frac 12}<L_3>^{\frac
12-\delta}}\nonumber\\
&\lesssim &\sum_{L_1, L_2, L_3\gtrsim 1} \frac{L_{min}^{\frac
12}}{<L_1>^{\frac 12}<L_2>^{\frac 12}<L_3>^{\frac
12-\delta}}\nonumber\\
&\lesssim &\sum_{L_{min}, L_{med}, L_{max}\gtrsim 1}
\frac{1}{<L_{med}>^{\frac 12}<L_{max}>^{\frac 12-\delta}}\lesssim
1,\end{eqnarray} here we have used the estimate (without loss of
generality we assume $L_1\lesssim L_2\lesssim L_3$)
$$\left\|\Pi_{j=1}^3\chi_{|\xi_j|\lesssim 1}\chi_{|\lambda_1|\sim
L_1}\right\|_{[3,\R\times \R]}\lesssim
\left\|\chi_{|\xi_1|\lesssim 1}\left\|\chi_{|i\tau_1-h_1(\xi)|\sim
L_1}\right\|_{[3,\R]}\right\|_{[3,\R]}\lesssim L_1^{\frac 12}.$$
What remains is to estimate the term
\begin{equation}\left\|\frac{\xi_3 <\xi_1>^{\rho}<\xi_2>^{\rho}<\xi_3>^{-\rho}\chi_{\max\{|\xi_1|,|\xi_2|,|\xi_3|\}\gtrsim
1}}{<\lambda_1>^{\frac 12}<\lambda_2>^{\frac 12}<\lambda_3>^{\frac
12-\delta}}\right\|_{[3,\R\times \R]}.\label{3.15a}\end{equation}

 By taking the homogenous dyadic decomposition of the variable $|\xi_j|\sim N_j>0$,
by taking the non-homogenous dyadic decomposition of the variable
$|\lambda_j|\sim L_j\ge 1$, and the function $|h(\xi)|\sim H\ge 1$
( here the notation $|\lambda_j|\sim L_j=1$, $|h(\xi)|\sim H=1$
means $|\lambda_j|\le 1$, $|h(\xi)|\le 1$, respectively), we have
\begin{eqnarray}&{}&\left\|\frac{\xi_3 <\xi_1>^{\rho}<\xi_2>^{\rho}<\xi_3>^{-\rho}\chi_{\max\{|\xi_1|,|\xi_2|,|\xi_3|\}
\gtrsim 1}}{<\lambda_1>^{\frac 12}<\lambda_2>^{\frac
12}<\lambda_3>^{\frac
12-\delta}}\right\|_{[3,\R\times \R]}\nonumber\\
&\lesssim& \left\|\sum_{N_{max}\gtrsim 1}\sum_{L_1,L_2,L_3\ge
1}\sum_{H\ge 1}\frac{N_3
<N_1>^{\rho}<N_2>^{\rho}}{<N_3>^{\rho}L_1^{\frac 12}L_2^{\frac
12}L_3^{\frac
12-\delta}}X_{N_1,N_2,N_3;H;L_1,L_2,L_3}\right\|_{[3,\R\times
\R]},\label{2}\end{eqnarray} where $X_{N_1,N_2,N_3;H;L_1,L_2,L_3}$
is the multiplier
\begin{equation}X_{N_1,N_2,N_3;H;L_1,L_2,L_3}
:=\chi_{|h(\xi)|\thicksim H}\Pi_{j=1}^3\chi_{|\xi_j|\sim N_j}
\chi_{|\lambda_j|\sim L_j}.\nonumber\end{equation} From the
identities $\xi_1+\xi_2+\xi_3=0$ and $\tau_1+\tau_2+\tau_3=0$ we
see that
$$h(\xi)=-\lambda_1-\lambda_2-\lambda_3=3i\xi_1\xi_2\xi_3-(|\xi_1|^{2\alpha}+|\xi_2|^{2\alpha}+|\xi_3|^{2\alpha}).
$$ Then  the multiplier
$X_{N_1,N_2,N_3;H;L_1,L_2,L_3}$ vanishes unless
\begin{equation}N_{max}\sim N_{med}, L_{max}\sim \max\{H,L_{med}\}, H\sim \max\{N_{max}^2N_{min},N_{max}^{2\alpha}\}.
 \label{3}\end{equation}
Thus we may implicitly assume (\ref{3}) in the summations. By
applying Schur's test (Lemma 3.11 in \cite{Tao}),
\begin{eqnarray}(\ref{2})&\lesssim& \sup_{N\gtrsim 1}\left\| \sum_{N_{max}\sim N_{med}\sim N}\, \sum_{H\ge 1}
\right.\nonumber\\
&{}&\qquad\left. \sum_{L_{max}\sim \max\{H,L_{med}\}}\frac{N_3
<N_1>^{\rho}<N_2>^{\rho}}{<N_3>^{\rho}L_1^{\frac 12}L_2^{\frac
12}L_3^{\frac
12-\delta}}X_{N_1,N_2,N_3;H;L_1,L_2,L_3}\right\|_{[3,\R\times
\R]}.\label{5}\end{eqnarray} In light of (\ref{3}) and the
comparison principle in \cite{Tao},  we thus see that at least one
of the inequalities
\begin{equation}(\ref{5})\lesssim \sup_{N\gtrsim
1}\sum_{N_{max}\sim N_{med}\sim N} \, \sum_{L_{max}\gtrsim
L_{med}\gtrsim L_{min}}\frac{N_3
<N_1>^{\rho}<N_2>^{\rho}}{<N_3>^{\rho}L_1^{\frac 12}L_2^{\frac
12}L_3^{\frac 12-\delta}}\left\|X_{N_1,N_2,N_3; L_{max};
L_1,L_2,L_3}\right\|_{[3,R\times R]},\label{6}\end{equation} or
\begin{equation}(\ref{5})\lesssim \sup_{N\gtrsim 1}\sum_{ N_{max}\sim N_{med}\sim N}
\,\sum_{L_{max}\sim L_{med}}\,\sum_{H\ll L_{max}} \frac{N_3
<N_1>^{\rho}<N_2>^{\rho}}{<N_3>^{\rho}L_1^{\frac 12}L_2^{\frac
12}L_3^{\frac
12-\delta}}\left\|X_{N_1,N_2,N_3;H;L_1,L_2,L_3}\right\|_{[3,\R\times
\R]}\label{7}\end{equation} holds.  It is sufficient to prove
$(\ref{6})\lesssim 1$ and $(\ref{7})\lesssim 1$.

 {\bf The proof of
$(\ref{7})\lesssim 1$.}\quad Note that the inequality
$N_{max}^2N_{min}\ge N_{max}^{2\alpha}$ implies $N_{min}\ge
N_{max}^{2\alpha-2}$. When  $N_{min}\gtrsim 1$, by  using the
estimate (1) in Lemma \ref{L3.1}, we get from (\ref{2}) and
(\ref{3}),
\begin{eqnarray}(\ref{7})&\lesssim&\sup_{N\gtrsim 1}\sum_{ N_{max}\sim N_{med}\sim
N}\, \sum_{L_{max}\sim L_{med}}\nonumber\\
&{}&\qquad \sum_{
H\sim\max\{N_{max}^2N_{min},N_{max}^{2\alpha}\}\ll L_{max}}
\frac{N_3 <N_1>^{\rho}<N_2>^{\rho}}{<N_3>^{\rho}L_1^{\frac
12}L_2^{\frac 12}L_3^{\frac
12-\delta}}L_{min}^{\frac 12}N_{min}^{\frac 12}\nonumber\\
&\lesssim&\sup_{N\gtrsim 1}\sum_{ N_{max}\sim N_{med}\sim
N,N_{min}\gtrsim 1}\,\sum_{L_{max}\sim L_{med}}\,\sum_{H\sim
N_{max}^2N_{min}\ll L_{max}} \frac{N_{min}^{1-\rho}
N^{2\rho}}{L_{min}^{\frac
12}L_{max}^{1-\delta}}L_{min}^{\frac 12}N_{min}^{\frac 12}\nonumber\\
&\lesssim&\sup_{N\gtrsim 1}\sum_{ N_{max}\sim N_{med}\sim
N,N_{min}\gtrsim 1}\,\sum_{L_{max}\sim L_{med}\gtrsim
N^2N_{min}}\frac{N_{min}^{\frac 12-\rho+2\delta}
N^{2\rho-2+4\delta}}{L_{max}^{\delta}}log _2(L_{max})\nonumber\\
&\lesssim&\sup_{N\gtrsim 1}\sum_{ N_{max}\sim N_{med}\sim
N,N_{min}\gtrsim 1}N_{min}^{\frac 12-\rho+2\delta}
N^{2\rho-2+4\delta}\nonumber\\
&\lesssim&\sup_{N\gtrsim 1}N^{2\rho-2+4\delta}\lesssim
1,\label{3.60}\end{eqnarray} for $\frac 12<\rho<1$ and $\delta>0$
small.

When $N_1\sim N_2\gg N_3$ with  $N_{3}\lesssim 1$, by  using the
estimate (1) in Lemma \ref{L3.1}, we get from (\ref{2}) and
(\ref{3}),
\begin{eqnarray}(\ref{7})&\lesssim&\sup_{N\gtrsim 1}\sum_{ N_{1}
\sim N_{2}\sim N, N_3\lesssim 1}\, \sum_{L_{max}\sim
L_{med}}\,\sum_{ H\sim\max\{N^2N_{3},N^{2\alpha}\}\ll L_{max}}
\frac{N_3 <N_1>^{\rho}<N_2>^{\rho}}{<N_3>^{\rho}L_1^{\frac
12}L_2^{\frac
12}L_3^{\frac 12-\delta}}L_{min}^{\frac 12}N_{3}^{\frac 12}\nonumber\\
&\lesssim&\sup_{N\gtrsim 1}\sum_{ N_{1}\sim N_{2}\sim N,1\gtrsim
N_{3}\ge N^{2\alpha-2}}\,\sum_{L_{max}\sim L_{med}}\,\sum_{H\sim
N^2N_{3}\ll L_{max}} \frac{N_{3} N^{2\rho}}{L_{min}^{\frac
12}L_{max}^{1-\delta}}L_{min}^{\frac 12}N_{3}^{\frac 12}\nonumber\\
&{}&+\sup_{N\gtrsim 1}\sum_{ N_{1}\sim N_{2}\sim N,N_{3}\le
N^{2\alpha-2}}\,\sum_{L_{max}\sim L_{med}}\,\sum_{H\sim
 N^{2\alpha}\ll L_{max}} \frac{N_{3}
N^{2\rho}}{L_{min}^{\frac
12}L_{max}^{1-\delta}}L_{min}^{\frac 12}N_3^{\frac 12}\nonumber\\
&\lesssim&\sup_{N\gtrsim 1}\sum_{ N_{1}\sim N_{2}\sim N,1\gtrsim
N_{3}\ge N^{2\alpha-2}}\,\sum_{L_{max}\sim L_{med}\gtrsim
N^2N_{3}}\frac{N_{3}^{\frac 12+2\delta}
N^{2\rho-2+4\delta}}{L_{max}^{\delta}}log _2(L_{max})\nonumber\\
&{}&+\sup_{N\gtrsim 1}\sum_{ N_{1}\sim N_{2}\sim N,N_{3}\le
N^{2\alpha-2}}\,\sum_{L_{max}\sim L_{med}\gtrsim N^{2\alpha}}
\frac{N_{3}^{\frac 32}
N^{2\rho-2\alpha+2\alpha\delta}}{L_{max}^{\delta}}log _2(L_{max})\nonumber\\
&\lesssim&\sup_{N\gtrsim 1}\sum_{ N_{1}\sim N_{2}\sim N,1\gtrsim
N_{3}\ge N^{2\alpha-2}}N_{3}^{\frac 12+2\delta}
N^{2\rho-2+4\delta}\nonumber\\
&{}&\quad +\sup_{N\gtrsim 1}\sum_{ N_{1}\sim N_{2}\sim N,N_{3}\le
N^{2\alpha-2}}N_{3}^{\frac 32}
N^{2\rho-2\alpha+2\alpha\delta}\nonumber\\
&\lesssim&\sup_{N\gtrsim 1}N^{2\rho-2+4\delta}+\sup_{N\gtrsim 1}
N^{2\rho+\alpha-3+2\delta}\lesssim 1,\label{3.61}\end{eqnarray}
for $\delta>0$ small, since $\frac 12<\rho<1$ and $0\le\alpha\le
1$ imply
$$2\rho+\alpha-3+2\delta<0,\, 2\rho-2+4\delta<0$$ for $\delta>0$ small.

When $\alpha=1$ and $N_1\sim N_3\sim N\gg N_2$ with $N_{2}\lesssim
1$,  we have $N^{2\alpha}\ge N^2N_{min}$. By using the estimate
(1) in Lemma \ref{L3.1}, we get from (\ref{2}) and (\ref{3}),
\begin{eqnarray}(\ref{7})&\lesssim&\sup_{N\gtrsim 1}
\sum_{ N_{1}\sim N_{3}\sim N, N_2\lesssim 1}\,\sum_{L_{max}\sim
L_{med}}\, \sum_{ H\sim N^{2\alpha}\ll L_{max}}
\frac{N}{L_1^{\frac 12}L_2^{\frac 12}L_3^{\frac
12-\delta}}L_{min}^{\frac 12}N_{2}^{\frac 12}\nonumber\\
&\lesssim&\sup_{N\gtrsim 1}\sum_{N_{1}\sim N_{3}\sim N,
N_2\lesssim 1}\, \sum_{L_{max}\sim L_{med}\gg  N^{2}}
\frac{N}{L_1^{\frac 12}L_2^{\frac 12}L_3^{\frac
12-\delta}}L_{min}^{\frac 12}N_{2}^{\frac 12}\nonumber\\
&\lesssim&\sup_{N\gtrsim 1}\sum_{N_2\lesssim 1}\,
\sum_{L_{max}\sim L_{med}\gg  N^{2}}
 \frac{NN_2^{\frac
12}}{L_{med}^{\frac 12}L_{max}^{\frac 12-\delta}}\nonumber\\
&\lesssim&\sup_{N\gtrsim 1}\sum_{N_2\lesssim 1}\,
\sum_{L_{max}\sim L_{med}\gg  N^{2}}
 \frac{N^{4\delta-1}N_2^{\frac
12}}{L_{max}^{\delta}}\nonumber\\
&\lesssim& \sup_{N\gtrsim 1}\sum_{N_2\lesssim
1}N^{4\delta-1}N_2^{\frac 12}\lesssim 1\label{3.73}\end{eqnarray}
for $\delta>0$ small.

When $0\le\alpha<1$ and $N_1\sim N_3\sim N\gg N_2$ with
$N_{2}\lesssim 1$, by using the estimate (1) in Lemma \ref{L3.1},
we get from (\ref{2}) and (\ref{3}),
\begin{eqnarray}(\ref{7})&\lesssim&\sup_{N\gtrsim 1}\sum_{ N_{1}\sim N_{3}\sim
N}\,\sum_{L_{max}\sim L_{med}}\, \sum_{
H\sim\max\{N^2N_{2},N^{2\alpha}\}\ll L_{max}} \frac{N}{L_1^{\frac
12}L_2^{\frac 12}L_3^{\frac
12-\delta}}L_{min}^{\frac 12}N_{2}^{\frac 12}\nonumber\\
&\lesssim&\sup_{N\gtrsim 1}\sum_{N_{1}\sim N_{3}\sim N, N_2\ge
N^{2\alpha-2}}\, \sum_{L_{max}\sim L_{med}\gg  N^2N_{2}}
\frac{N}{L_1^{\frac 12}L_2^{\frac 12}L_3^{\frac
12-\delta}}L_{min}^{\frac 12}N_{2}^{\frac 12}\nonumber\\
&{}& + \sup_{N\gtrsim 1}\sum_{N_{1}\sim N_{3}\sim N, N_2\le
N^{2\alpha-2}}\, \sum_{L_{max}\sim L_{med}\gg  N^{2\alpha}}
\frac{N}{L_1^{\frac 12}L_2^{\frac 12}L_3^{\frac
12-\delta}}L_{min}^{\frac 12}N_{2}^{\frac 12}\nonumber\\
&\lesssim&\sup_{N\gtrsim 1}\sum_{N_2\ge N^{2\alpha-2}}\,
\sum_{L_{max}\sim L_{med}\gg  N^2N_{2}} \frac{NN_2^{\frac
12}}{L_{med}^{\frac 12}L_{max}^{\frac 12-\delta}}
\nonumber\\
&{}& + \sup_{N\gtrsim 1}\sum_{N_2\le N^{2\alpha-2}}\,
\sum_{L_{max}\sim L_{med}\gg  N^{2\alpha}}
 \frac{NN_2^{\frac
12}}{L_{med}^{\frac 12}L_{max}^{\frac 12-\delta}}\nonumber\\
&\lesssim&\sup_{N\gtrsim 1}\sum_{N_2\ge N^{2\alpha-2}}\,
\sum_{L_{max}\sim L_{med}\gg  N^2N_{2}}
\frac{N^{-1+4\delta}N_2^{-\frac 12+2\delta}}{L_{max}^{\delta}}
\nonumber\\
&{}& + \sup_{N\gtrsim 1}\sum_{N_2\le N^{2\alpha-2}}\,
\sum_{L_{max}\sim L_{med}\gg  N^{2\alpha}}
 \frac{N^{4\alpha\delta-2\alpha+1}N_2^{\frac
12}}{L_{max}^{\delta}}\nonumber\\
&\lesssim&\sup_{N\gtrsim 1}\sum_{N_2\ge N^{2\alpha-2}}
N^{-1+4\delta}N_2^{-\frac 12+2\delta}+ \sup_{N\gtrsim
1}\sum_{N_2\le N^{2\alpha-2}}N^{4\alpha\delta-2\alpha+1}N_2^{\frac
12}\nonumber\\
&\lesssim&\sup_{N\gtrsim 1} N^{-\alpha+4\alpha\delta}\lesssim
1\label{3.1e}\end{eqnarray} for $\delta>0$ small. By symmetric we
know the estimate  $(\ref{7})\lesssim 1$  holds when  $N_2\sim
N_3\sim N\gg N_1$ and $N_{1}\lesssim 1$.

{\bf The proof of $(\ref{6})\lesssim 1$.} \quad  We first deal
with the contribution where the case (2a)  in Lemma \ref{L3.1}
holds. In this case we have $N_1\sim N_2\sim N_3\sim N$,
$L_{max}\sim N^3$ and $L_{min}\gtrsim N^{2\alpha}$, since we have
$L_j\sim |\lambda_j|\ge |\xi_j|^{2\alpha}$. So we get
\begin{eqnarray}(\ref{6}) &\lesssim&\sup_{N\gtrsim 1}
\sum_{L_{med}\gtrsim N^{2\alpha}, L_{max}\sim
N^3}\frac{N^{1+\rho}}{L_1^{\frac 12}L_2^{\frac 12}L_3^{\frac
12-\delta}}L_{min}^{\frac 12}\min\{N^{-\frac 14}L_{med}^{\frac
14},L_{med}^{\frac
1{4\alpha}}\}\nonumber\\
&\lesssim&\sup_{N\gtrsim 1} \sum_{L_{med}\gtrsim N^{2\alpha},
L_{max}\sim N^3}\frac{N^{\frac 34+\rho}}{L_{med}^{\frac
14}L_{max}^{\frac 12-\delta}}\lesssim \sup_{N\gtrsim 1}
\sum_{L_{max}\sim N^3}\frac{N^{-\frac
34+\rho-\frac{\alpha}2+6\delta}}{L_{max}^{\delta}}\nonumber\\
&\lesssim&\sup_{N\gtrsim 1}N^{-\frac
34+\rho-\frac{\alpha}2+6\delta}\lesssim 1\label{3.2}
\end{eqnarray}
for $\delta>0$ small, since we have $\rho<\frac {3+2\alpha}4$.

Second,  we deal with the contribution where the case (2b) in
Lemma \ref{L3.1}  applies. We choose $\beta>0$ small in
(\ref{13}). We do not have perfect symmetry and must consider the
cases
$$\mbox{\bf{Case A:}}\quad N\sim N_1\sim
N_2\gg N_3\gtrsim 1;  H\sim L_3\gtrsim L_1,L_2,$$
$$\mbox{\bf{Case B:}}\quad N\sim N_1\sim
N_2\gg N_3, N_3\lesssim 1;  H\sim L_3\gtrsim L_1,L_2,$$
$$\mbox{\bf{Case C:}}\quad N\sim N_1\sim
N_3\gg N_2\gtrsim 1;  H\sim L_2\gtrsim L_1,L_3,$$
$$\mbox{\bf{Case D:}}\quad N\sim N_1\sim
N_3\gg N_2, N_2\lesssim 1;  H\sim L_2\gtrsim L_1,L_3,$$separately.

{\bf The estimate in Case A.} \quad In this case, we have
$L_{max}\sim N^2N_3$ and $L_{med}\gtrsim N^{2\alpha}$, and then
$N_3^{\frac 12}\ll N^{\frac 12}\le L_{med}^{\frac 1{2\alpha}}$.
When $0\le \alpha<1$, we have $N_3^{\beta+1}N^{2-\beta}\ge
N^{2\alpha}$  for $N_3\gtrsim 1$ and $\beta>0$ small.

When $L_{med}\ge N_3^{\beta+1}N^{2-\beta}\ge N^{2\alpha}$,  we get
from (\ref{13}) and (\ref{6}) that
\begin{eqnarray}(\ref{6}) &\lesssim&\sup_{N\gtrsim 1}\sum_{N\gg N_3\gtrsim 1}
\sum_{L_{max}\sim N^2N_3\gtrsim L_{med}\ge
N_3^{\beta+1}N^{2-\beta}}\frac{N_3^{1-\rho}N^{2\rho}}{L_1^{\frac
12}L_2^{\frac 12}L_3^{\frac 12-\delta}}L_{min}^{\frac
12}N_3^{\frac 12}\nonumber\\
&\lesssim&\sup_{N\gtrsim 1}\sum_{N\gg N_3\gtrsim 1}
\sum_{L_{max}\sim N^2N_3,L_{med}\ge
N_3^{\beta+1}N^{2-\beta}}\frac{N_3^{1-\rho+\delta}N^{2\rho-1+2\delta}}{L_{med}^{\frac
12}}\nonumber\\
&\lesssim&\sup_{N\gtrsim 1}\sum_{N\gg N_3\gtrsim 1}
\,\sum_{L_{med}\lesssim L_{max}\sim N^2N_3}\frac {N_3^{\frac
12-\frac{\beta}2-\rho+\delta(2+\beta)}
N^{2\rho-2+\frac{\beta}2+\delta(4-\beta)}}{L_{med}^{\delta}}\nonumber\\
&\lesssim&\sup_{N\gtrsim 1}\sum_{N\gg N_3\gtrsim 1} N_3^{\frac
12-\frac{\beta}2-\rho+\delta(2+\beta)}
N^{2\rho-2+\frac{\beta}2+\delta(4-\beta)}\nonumber\\
&\lesssim&\sup_{N\gtrsim
1}N^{2\rho-2+\frac{\beta}2+\delta(4-\beta)}\lesssim
1,\label{3.63}\end{eqnarray} for $\delta>0$ and $\beta>0$
 small, since the
inequality $\frac 12<\rho<1$ implies
$$2\rho-2+\frac{\beta}2+\delta(4-\beta)<0,\,\frac
12-\frac{\beta}2-\rho+\delta(2+\beta)<0 $$  for  $\delta>0$ and
$\beta>0$  small.

When $L_{med}\le N_3^{\beta+1}N^{2-\beta}$ and $L_{med}\gtrsim
N^{2\alpha}$, We get from (\ref{13}) and (\ref{6}) that
\begin{eqnarray}(\ref{6}) &\lesssim&\sup_{N\gtrsim 1}\sum_{ N_3\gtrsim 1} \sum_{L_{med}\le
N_3^{\beta+1}N^{2-\beta},L_{3}\sim
N^2N_3}\frac{N_3^{1-\rho}N^{2\rho}}{L_1^{\frac 12}L_2^{\frac
12}L_3^{\frac 12-\delta}}L_{min}^{\frac
12}N_3^{-\frac 1{2\beta}}N^{\frac{\beta-2}{2\beta}}L_{med}^{\frac 1{2\beta}}\nonumber\\
&\lesssim&\sup_{N\gtrsim 1}\sum_{ N_3\gtrsim 1} \sum_{L_{max}\sim
N^2N_3}\frac{N_3^{\frac 12-\rho+(2+\beta)\delta-\frac
{\beta}2}N^{2\rho-2+2\delta+\frac{\beta}2+(2-\beta)\delta}}{L_{med}^{\delta}}\nonumber\\
&\lesssim&\sup_{N\gtrsim 1}\sum_{N_3\gtrsim 1} N_3^{\frac
12-\rho+(2+\beta)\delta-\frac
{\beta}2}N^{2\rho-2+2\delta+\frac{\beta}2+(2-\beta)\delta}\lesssim
1\label{3.65}\end{eqnarray} for $\delta>0$ and $\beta>0$ small,
since the inequality $\frac 12<\rho<1$ means
$$\frac
12-\rho+(2+\beta)\delta-\frac {\beta}2<0,\,
2\rho-2+2\delta+\frac{\beta}2+(2-\beta)\delta<0$$  for  $\delta>0$
and $\beta>0$  small.

When $\alpha=1$, we must consider the case  $L_{med}\ge
N_3^{\beta+1}N^{2-\beta}$ , $L_{med}\gtrsim N^{2}$ and
$N_3^{\beta+1}N^{2-\beta}\le N^{2}$. We have $1\lesssim N_3\le
N^{\frac{\beta}{\beta+1}}$. We get from (\ref{13}) and (\ref{6})
that
\begin{eqnarray}(\ref{6}) &\lesssim&\sup_{N\gtrsim 1}
\sum_{N_3\le N^{\frac{\beta}{\beta+1}}, N_3\lesssim 1} \,\sum_{
L_{3}\sim N^2N_3,L_{med}\ge N^{2}}\frac{N_3N^{2\rho}}{L_1^{\frac
12}L_2^{\frac 12}L_3^{\frac 12-\delta}}L_{min}^{\frac
12}N_3^{\frac 12}\nonumber\\
&\lesssim&\sup_{N\gtrsim 1}\sum_{N_3\le N^{\frac{\beta}{\beta+1}}}
\sum_{L_{3}\sim N^2N_3,L_{med}\ge
N^{2}}\frac{N_3^{1+\delta}N^{2\rho-1+2\delta}}{L_{med}^{\frac
12}}\nonumber\\
&\lesssim&\sup_{N\gtrsim 1}\sum_{ N_3\le
N^{\frac{\beta}{\beta+1}}} \sum_{L_{min}\le
L_{med}}\frac{N_3^{1+\delta}N^{2\rho-2+4\delta}}{L_{med}^{\delta}}\nonumber\\
&\lesssim& \sup_{N\gtrsim 1}\sum_{N_3\le
N^{\frac{\beta}{\beta+1}}}
N_3^{1+\delta}N^{2\rho-2+4\delta}\nonumber\\
&\lesssim& \sup_{N\gtrsim
1}N^{\frac{\beta}{\beta+1}(1+\delta)+2\rho-2+4\delta} \lesssim
1\label{3.76} \end{eqnarray} for  $\delta>0$ and $\beta>0$  small,
since the inequality $0\le \rho <1$ implies
$$\frac{\beta}{\beta+1}(1+\delta)+2\rho-2+4\delta<0$$  for  $\delta>0$ and $\beta>0$
small. We complete the estimate in Case A.

 {\bf The estimate in Case B.} \quad In this case, we have $L_{max}\sim
N^2N_3$ and $L_{med}\gtrsim N^{2\alpha}$, and then $N_3^{\frac
12}\ll N^{\frac 12}\le L_{med}^{\frac 1{2\alpha}}$.

When $L_{med}\ge N_3^{\beta+1}N^{2-\beta}\ge N^{2\alpha}$, we have
$N_3\ge N^{\frac{2\alpha-2+\beta}{\beta+1}}$. By using
$N_3\lesssim 1$ we get from (\ref{13}) and (\ref{6}) that
\begin{eqnarray}(\ref{6}) &\lesssim&\sup_{N\gtrsim 1}\sum_{N_1\sim N_2\sim N,
N_3\lesssim 1} \sum_{L_{max}\sim N^2N_3\gtrsim L_{med}\ge
N_3^{\beta+1}N^{2-\beta}}\frac{N_3N^{2\rho}}{L_1^{\frac
12}L_2^{\frac 12}L_3^{\frac 12-\delta}}L_{min}^{\frac
12}N_3^{\frac 12}\nonumber\\
&\lesssim&\sup_{N\gtrsim 1}\sum_{N_3\lesssim 1} \sum_{L_{max}\sim
N^2N_3,L_{med}\ge
N_3^{\beta+1}N^{2-\beta}}\frac{N_3^{1+\delta}N^{2\rho-1+2\delta}}{L_{med}^{\frac
12}}\nonumber\\
&\lesssim&\sup_{N\gtrsim 1}\sum_{N_3\lesssim 1}
\,\sum_{L_{med}\lesssim L_{max}\sim N^2N_3}\frac {N_3^{\frac
12-\frac{\beta}2+\delta(2+\beta)}
N^{2\rho-2+\frac{\beta}2+\delta(4-\beta)}}{L_{med}^{\delta}}\nonumber\\
&\lesssim&\sup_{N\gtrsim 1}\sum_{N_3\lesssim 1} N_3^{\frac
12-\frac{\beta}2+\delta(2+\beta)}
N^{2\rho-2+\frac{\beta}2+\delta(4-\beta)}\nonumber\\
&\lesssim&\sup_{N\gtrsim 1}
N^{2\rho-2+\frac{\beta}2+\delta(4-\beta)}\lesssim
1,\label{3.66}\end{eqnarray} for $\delta>0$ and $\beta>0$
 small, since the
inequality $\frac 12<\rho<1$ implies
$$2\rho-2+\frac{\beta}2+\delta(4-\beta)<0,\, \frac
12-\frac{\beta}2+\delta(2+\beta)>0$$  for  $\delta>0$ and
$\beta>0$ small.

When $L_{med}\ge N_3^{\beta+1}N^{2-\beta}$ , $L_{med}\gtrsim
N^{2\alpha}$ and $N_3^{\beta+1}N^{2-\beta}\le N^{2\alpha}$, we
have $N_3\le N^{\frac{2\alpha-2+\beta}{\beta+1}}$. We get from
(\ref{13}) and (\ref{6}) that
\begin{eqnarray}(\ref{6}) &\lesssim&\sup_{N\gtrsim 1}\sum_{N_3\le N^{\frac{2\alpha-2+\beta}{\beta+1}}, N_3\lesssim 1}
\,\sum_{ L_{3}\sim N^2N_3,L_{med}\ge
N^{2\alpha}}\frac{N_3N^{2\rho}}{L_1^{\frac 12}L_2^{\frac
12}L_3^{\frac 12-\delta}}L_{min}^{\frac
12}N_3^{\frac 12}\nonumber\\
&\lesssim&\sup_{N\gtrsim 1}\sum_{N_3\le
N^{\frac{2\alpha-2+\beta}{\beta+1}}} \sum_{L_{3}\sim
N^2N_3,L_{med}\ge
N^{2\alpha}}\frac{N_3^{1+\delta}N^{2\rho-1+2\delta}}{L_{med}^{\frac
12}}\nonumber\\
&\lesssim&\sup_{N\gtrsim 1}\sum_{ N_3\le
N^{\frac{2\alpha-2+\beta}{\beta+1}}} \sum_{L_{min}\le
L_{med}}\frac{N_3^{1+\delta}N^{2\rho-1+2\delta-(1-2\delta)\alpha}}{L_{med}^{\delta}}\nonumber\\
&\lesssim& \sup_{N\gtrsim 1}\sum_{N_3\le
N^{\frac{2\alpha-2+\beta}{\beta+1}}}
N_3^{1+\delta}N^{2\rho-1+2\delta-(1-2\delta)\alpha}\nonumber\\
&\lesssim& \sup_{N\gtrsim
1}N^{\frac{2\alpha-2+\beta}{\beta+1}(1+\delta)+2\rho-1+2\delta-(1-2\delta)\alpha}
\lesssim 1\label{3.67} \end{eqnarray} for  $\delta>0$ and
$\beta>0$  small, since the inequality $\frac 12<\rho<1$ implies
$$\frac{2\alpha-2+\beta}{\beta+1}(1+\delta)+2\rho-1+2\delta-(1-2\delta)\alpha<0$$  for  $\delta>0$ and $\beta>0$
 small.

When $L_{med}\le N_3^{\beta+1}N^{2-\beta}$ and $L_{med}\gtrsim
N^{2\alpha}$, we have $1\gtrsim N_3\ge
N^{\frac{2\alpha-2+\beta}{\beta+1}}$. We get from (\ref{13}) and
(\ref{6}) that
\begin{eqnarray}(\ref{6}) &\lesssim&\sup_{N\gtrsim 1}
\sum_{N_3\gtrsim 1 } \sum_{L_{med}\le
N_3^{\beta+1}N^{2-\beta},L_{3}\sim
N^2N_3}\frac{N_3N^{2\rho}}{L_1^{\frac 12}L_2^{\frac 12}L_3^{\frac
12-\delta}}L_{min}^{\frac
12}N_3^{-\frac 1{2\beta}}N^{\frac{\beta-2}{2\beta}}L_{med}^{\frac 1{2\beta}}\nonumber\\
&\lesssim&\sup_{N\gtrsim 1}\sum_{N_3\gtrsim 1} \sum_{L_{max}\sim
N^2N_3}\frac{N_3^{\frac 12+(2+\beta)\delta-\frac
{\beta}2}N^{2\rho-2+2\delta+\frac{\beta}2+(2-\beta)\delta}}{L_{med}^{\delta}}\nonumber\\
&\lesssim&\sup_{N\gtrsim 1}\sum_{N_3\gtrsim 1} N_3^{\frac
12+(2+\beta)\delta-\frac
{\beta}2}N^{2\rho-2+2\delta+\frac{\beta}2+(2-\beta)\delta}\nonumber\\
&\lesssim&\sup_{N\gtrsim
1}N^{2\rho-2+2\delta+\frac{\beta}2+(2-\beta)\delta} \lesssim
1\label{3.68}\end{eqnarray} for $\delta>0$ and $\beta>0$ small,
since the inequality $\frac 12<\rho<1$ means
$$2\rho-2+2\delta+\frac{\beta}2+(2-\beta)\delta<0,\, \frac
12+(2+\beta)\delta-\frac {\beta}2>0 $$  for  $\delta>0$ and
$\beta>0$  small. We complete the estimate in Case B.

{\bf The estimate in Case C.} \quad In this case, we have
$L_{max}\sim N^2N_2$ and $L_{med}\gtrsim N^{2\alpha}$, and then
$N_2^{\frac 12}\le N^{\frac 12}\le L_{med}^{\frac 1{2\alpha}}$. We
get from (\ref{13})   and (\ref{6}) that
\begin{eqnarray}(\ref{6}) &\lesssim&\sup_{N\gtrsim 1}\sum_{N\sim N_1\sim N_3\gg N_2\gtrsim 1}
\sum_{L_{1}, L_3\le L_{2}\sim
N^2N_2}\frac{N_3^{1-\rho}N_1^{\rho}N_2^{\rho}}{L_1^{\frac
12}L_2^{\frac 12}L_3^{\frac
12-\delta}}\left\|X_{N_1,N_2,N_3;L_{max};L_1,L_2,L_3}\right\|_{[3,\R\times
\R]}\nonumber\\
&\lesssim&\sup_{N\gtrsim 1}\sum_{N\sim N_1\sim N_3\gg N_2\gtrsim
1} \sum_{L_{1}, L_3\le L_{2}\sim
N^2N_2}\frac{N_3^{\rho}N_1^{\rho}N_2^{1-\rho}}{L_1^{\frac
12}L_2^{\frac 12-\delta}L_3^{\frac
12}}\left\|X_{N_1,N_2,N_3;L_{max};L_1,L_2,L_3}\right\|_{[3,\R\times
\R]}.\nonumber\end{eqnarray}
 By symmetry and the estimate obtained in Case A we get
\begin{equation}(\ref{6}) \lesssim \sup_{N\gtrsim 1}\sum_{N\sim N_1\sim N_3\gg N_2}
\sum_{L_{1}, L_3\le L_{2}\sim
N^2N_2}\frac{N_3^{1-\rho}N_1^{\rho}N_2^{\rho}}{L_1^{\frac
12}L_2^{\frac 12}L_3^{\frac
12-\delta}}\left\|X_{N_1,N_2,N_3;H;L_1,L_2,L_3}\right\|_{[3,R\times
R]}\lesssim 1. \label{21a}\end{equation}
 We complete the estimate in Case C.

 {\bf The estimate in Case D.} \quad In this case, we have $L_{max}=L_{2}\sim
N^2N_2$ and $L_{med}\gtrsim N^{2\alpha}$, and then $\alpha<1$. We
get from (\ref{13}) and (\ref{6}) that
\begin{eqnarray}(\ref{6}) &\lesssim&\sup_{N\gtrsim 1}\sum_{N_1\sim N_3\sim N, N_2\lesssim 1}
\,\sum_{L_{max}\sim N^2N_2,L_{med}\ge N^{2\alpha}}
\frac{N}{L_1^{\frac 12}L_2^{\frac 12}L_3^{\frac
12-\delta}}L_{min}^{\frac
12}N_2^{\frac 12}\nonumber\\
&\lesssim&\sup_{N\gtrsim 1}\sum_{N_1\sim N_3\sim N,N_2\lesssim 1}
\,\sum_{L_{max}\sim N^2N_2,L_{med}\ge N^{2\alpha}}
\frac{NN_2^{\frac
12}}{L_{max}^{\frac 12}L_{med}^{\frac 12-\delta}}\nonumber\\
&\lesssim&\sup_{N\gtrsim 1}\sum_{N_1\sim N_3\sim N, N_2\lesssim 1}
\,\sum_{L_{max}\sim
N^2N_2,L_{med}\ge N^{2\alpha}} \frac{L_{max}^{\delta}}{L_{med}^{\frac 12}}\nonumber\\
&\lesssim&\sup_{N\gtrsim 1}\sum_{N_1\sim N_3\sim N, N_2\lesssim 1}
\,\sum_{L_{max}\sim
N^2N_2,L_{med}\ge N^{2\alpha}} \frac{N_2^{\delta}N^{-\alpha+2\delta(1+\alpha)}}{L_{med}^{\delta}}\nonumber\\
&\lesssim&\sup_{N\gtrsim 1}\sum_{N_2\lesssim
1}N_2^{\delta}N^{-\alpha+2\delta(1+\alpha)} \lesssim
\sup_{N\gtrsim 1} N^{-\alpha+2\delta(1+\alpha)}\lesssim
1,\label{19b} \end{eqnarray} for  $\delta>0$  small, since we have
$0\le\alpha<1$ in this case. We complete the estimate where the
case (2b) in Lemma \ref{L3.1} applies.

To finish the estimate of (\ref{6}) it remains to deal with the
case where (2C) in Lemma \ref{L3.1} holds.  When
$N_{min}=N_3\gtrsim 1$, we have $L_3\ll L_{max}$ and $N^2N_3\ge
N^{2\alpha}$, and then
\begin{eqnarray}(\ref{6}) &\lesssim&\sup_{N\gtrsim
1}\sum_{N_{med}\sim N_{max}\sim N,N_{min}\gtrsim 1}\,\sum_{L_3\ll
L_{max}\sim N^2N_{min}}\frac{N_{min}^{1-\rho}N^{2\rho}}{L_1^{\frac
12}L_2^{\frac 12}L_3^{\frac 12-\delta}}L_{min}^{\frac
12}N^{-1}L_{med}^{\frac 12}\nonumber\\
&\lesssim&\sup_{N\gtrsim 1}\sum_{N_{med}\sim N_{max}\sim N,
N_{min}\gtrsim 1}\,\sum_{L_{max}\sim
N^2N_{min}}\frac{N_{min}^{1-\rho}N^{2\rho-1}}{L_{max}^{\frac
12}L_{min}^{\frac 12}L_{med}^{\frac 12-\delta}}L_{min}^{\frac
12}N^{-1}L_{med}^{\frac 12}\nonumber\\
&\lesssim&\sup_{N\gtrsim 1}\sum_{N_{med}\sim N_{max}\sim
N,N_{min}\gtrsim 1}\,\sum_{L_{max}\sim
N^2N_{min}}\frac{N_{min}^{1-\rho}N^{2\rho}}{L_{max}^{\frac
12}L_{max}^{\frac 12-\delta}}\nonumber\\
&\lesssim&\sup_{N\gtrsim 1}\sum_{N_{med}\sim N_{max}\sim N,
N_{min}\gtrsim 1}\,\sum_{L_{max}\sim
N^2N_{min}}\frac{N_{min}^{\frac
12-\rho+2\delta}N^{2\rho-2+4\delta}}{L_{max}^{\delta}}\nonumber\\
&\lesssim&\sup_{N\gtrsim 1}\sum_{N_{med}\sim N_{max}\sim N,
N_{min}\gtrsim 1}N_{min}^{\frac
12-\rho+2\delta}N^{2\rho-2+4\delta}\nonumber\\
&\lesssim&\sup_{N\gtrsim 1}N^{2\rho-2+4\delta}\lesssim
1,\label{3.69}\end{eqnarray} for $\delta>0$  small.

 When $N_{min}=N_3\lesssim 1$, we have $L_3\ll
L_{max}$, and
\begin{eqnarray}(\ref{6}) &\lesssim&\sup_{N\gtrsim
1}\sum_{N_{med}\sim N_{max}\sim N, N_{min}\lesssim
1}\,\sum_{L_3\ll L_{max}\sim \max\{N^2N_{min},
N^{2\alpha}\}}\frac{N_{min}N^{2\rho}}{L_1^{\frac 12}L_2^{\frac
12}L_3^{\frac 12-\delta}}L_{min}^{\frac
12}N^{-1}L_{med}^{\frac 12}\nonumber\\
&\lesssim&\sup_{N\gtrsim 1}\sum_{N_{med}\sim N_{max}\sim N,
N_{min}\lesssim 1}\,\sum_{L_{max}\sim \max\{N^2N_{min},
N^{2\alpha}\}}\frac{N_{min}N^{2\rho-1}}{L_{max}^{\frac
12}L_{min}^{\frac 12}L_{med}^{\frac 12-\delta}}L_{min}^{\frac
12}N^{-1}L_{med}^{\frac 12}\nonumber\\
&\lesssim&\sup_{N\gtrsim 1}\sum_{N_{med}\sim N_{max}\sim N,
N_{min}\lesssim 1}\,\sum_{L_{max}\sim \max\{N^2N_{min},
N^{2\alpha}\}}\frac{N_{min}N^{2\rho}}{L_{max}^{\frac
12}L_{max}^{\frac 12-\delta}}\nonumber\\
&\lesssim&\sup_{N\gtrsim 1}\sum_{N_{med}\sim N_{max}\sim N,
1\gtrsim N_{min}\ge N^{2\alpha-2}}\,\sum_{L_{max}\sim
N^2N_{min}}\frac{N_{min}N^{2\rho}}{L_{max}^{\frac
12}L_{max}^{\frac 12-\delta}}\nonumber\\
&{}& +\sup_{N\gtrsim 1}\sum_{N_{med}\sim N_{max}\sim N, N_{min}\le
N^{2\alpha-2}}\,\sum_{L_{max}\sim
N^{2\alpha}}\frac{N_{min}N^{2\rho}}{L_{max}^{\frac
12}L_{max}^{\frac 12-\delta}}\nonumber\\
&\lesssim&\sup_{N\gtrsim 1}\sum_{N_{med}\sim N_{max}\sim N,
1\gtrsim N_{min}\ge N^{2\alpha-2}}\,\sum_{L_{max}\sim
N^2N_{min}}\frac{N_{min}^{\frac
12+2\delta}N^{2\rho-2+4\delta}}{L_{max}^{\delta}}\nonumber\\
&{}& +\sup_{N\gtrsim 1}\sum_{N_{med}\sim N_{max}\sim N, N_{min}\le
N^{2\alpha-2}}\,\sum_{L_{max}\sim
N^{2\alpha}}\frac{N_{min}N^{2\rho-1-\alpha+2\alpha\delta}}{L_{max}^{\delta}}\nonumber\\
&\lesssim&\sup_{N\gtrsim 1}\sum_{N_{med}\sim N_{max}\sim N,
N_{min}\ge N^{2\alpha-2}}N_{min}^{\frac
12-\rho+2\delta}N^{2\rho-2+4\delta}\nonumber\\
&{}& +\sup_{N\gtrsim 1}\sum_{N_{med}\sim N_{max}\sim N, N_{min}\le
N^{2\alpha-2}}N_{min}N^{2\rho-1-\alpha+2\alpha\delta}\nonumber\\
&\lesssim&\sup_{N\gtrsim 1}N^{2\rho-2+4\delta}+\sup_{N\gtrsim
1}N^{2\rho-3+\alpha+2\alpha\delta}\lesssim
1,\label{3.70}\end{eqnarray} for $\delta>0$  small, since the
inequalities $\frac 12<\rho<1$  and $0\le\alpha\le 1$ imply
$$2\rho-3+\alpha+2\alpha\delta<0$$ for  $\delta>0$ small.

When $N_{min}=N_2\gtrsim 1$, we have $L_2\ll L_{max}$ and
$N^2N_{min}\gtrsim N^{2\alpha}$, and then
\begin{eqnarray}(\ref{6}) &\lesssim&\sup_{N\gtrsim
1}\sum_{N_{med}\sim N_{max}\sim N,N_{min}\gtrsim 1}\,\sum_{L_2\ll
L_{max}\sim N^2N_{min}}\frac{N_{min}^{\rho}N}{L_{min}^{\frac
12}L_{med}^{\frac 12}L_{max}^{\frac 12-\delta}}L_{min}^{\frac
12}N^{-1}L_{med}^{\frac 12}\nonumber\\
&\lesssim&\sup_{N\gtrsim 1}\sum_{N_{med}\sim N_{max}\sim
N,N_{min}\gtrsim 1}\,\sum_{L_2\ll L_{max}\sim
N^2N_{min}}\frac{N_{min}^{\rho}}{L_{max}^{\frac
12-\delta}}\nonumber\\
&\lesssim&\sup_{N\gtrsim 1}\sum_{N_{med}\sim N_{max}\sim N,
N_{min}\gtrsim 1}\,\sum_{L_{max}\sim
N^2N_{min}}\frac{N_{min}^{\rho-\frac
12+2\delta}N^{-1+2\delta}}{L_{max}^{\delta}}\nonumber\\
&\lesssim&\sup_{N\gtrsim 1}\sum_{N_{med}\sim N_{max}\sim N,
N_{min}\gtrsim 1}N_{min}^{\rho-\frac
12+2\delta}N^{-1+2\delta}\nonumber\\
&\lesssim&\sup_{N\gtrsim 1} N^{-1+2\delta} \lesssim
1,\label{3.71}\end{eqnarray} for $\delta>0$  small.

When $N_{min}=N_2\lesssim 1$, we have $L_2\ll L_{max}$. Note that
$N^{-1}L_{med}^{\frac 12}\le N_2^{\frac 12}$ implies $L_{med}\le
N^2N_2$. We get from (\ref{14}) and (\ref{6}) that
\begin{eqnarray}(\ref{6}) &\lesssim&\sup_{N\gtrsim
1}\sum_{N_{1}\sim N_{3}\sim N,N_2\lesssim 1}\,\sum_{L_2\ll
L_{max}\sim \max\{N^2N_{2}, N^{2\alpha}\}}\frac{N}{L_{1}^{\frac
12}L_{2}^{\frac 12}L_{3}^{\frac 12-\delta}}L_{2}^{\frac
12}\min\{N^{-1}L_{med}^{\frac 12}, N_2^{\frac 12}\}\nonumber\\
&\lesssim&\sup_{N\gtrsim 1}\sum_{N_{1}\sim N_{3}\sim N,N_2\lesssim
1}\,\sum_{L_2\ll L_{max}\sim \max\{N^2N_{2},
N^{2\alpha}\},L_{med}\le N^2N_2 }\frac{N}{L_{1}^{\frac
12}L_{2}^{\frac 12}L_{3}^{\frac 12-\delta}}L_{2}^{\frac
12}N^{-1}L_{med}^{\frac 12}\nonumber\\
&{}& +\sup_{N\gtrsim 1}\sum_{N_{1}\sim N_{3}\sim N,N_2\lesssim
1}\,\sum_{L_2\ll L_{max}\sim \max\{N^2N_{2},
N^{2\alpha}\},L_{med}\ge N^2N_2 }\frac{N}{L_{1}^{\frac
12}L_{2}^{\frac 12}L_{3}^{\frac 12-\delta}}L_{2}^{\frac
12}N_2^{\frac 12}\nonumber\\
&\lesssim&\sup_{N\gtrsim 1}\sum_{N_{1}\sim N_{3}\sim N,N_2\lesssim
1}\,\sum_{L_2\ll L_{max}\sim \max\{N^2N_{2},
N^{2\alpha}\},L_{med}\le N^2N_2
}\frac{N^{2\delta}N_2^{\delta}}{L_{max}^{\frac
12-\delta}L_{med}^{\delta}}\nonumber\\
&{}& +\sup_{N\gtrsim 1}\sum_{N_{1}\sim N_{3}\sim N,N_2\lesssim
1}\,\sum_{L_2\ll L_{max}\sim \max\{N^2N_{2},
N^{2\alpha}\},L_{med}\ge N^2N_2 }\frac{NN_2^{\frac
12}}{L_{max}^{\frac 12-\delta}L_{med}^{\delta}}\nonumber\\
&\lesssim &\sup_{N\gtrsim 1}\sum_{N_2\ge N^{2\alpha-2}}\,\sum_{
L_{max}\sim N^2N_{2},L_{med}\le N^2N_2
}\frac{N^{2\delta}N_2^{\delta}}{L_{max}^{\frac
12-\delta}L_{med}^{\delta}}\nonumber\\
&{}& +\sup_{N\gtrsim 1}\sum_{N_2\le N^{2\alpha-2}}\,\sum_{
L_{max}\sim N^{2\alpha},L_{med}\le N^2N_2
}\frac{N^{2\delta}N_2^{\delta}}{L_{max}^{\frac
12-\delta}L_{med}^{\delta}}\nonumber\\
 &{}& +\sup_{N\gtrsim 1}\sum_{N_1
\sim N_{3}\sim N,N_2\lesssim 1}\,\sum_{L_2\ll L_{max}\sim
\max\{N^2N_{2}, N^{2\alpha}\},L_{med}\ge N^2N_2 }\frac{NN_2^{\frac
12}}{L_{max}^{\frac 12-\delta}L_{med}^{\delta}}\nonumber\\
&\lesssim&\sup_{N\gtrsim 1}\sum_{N_2\ge
N^{2\alpha-2}}N^{4\delta-1}N_2^{2\delta-\frac 12}+\sup_{N\gtrsim
1}\sum_{N_2\le
N^{2\alpha-2}}N^{-\alpha+2\delta(\alpha+1)}N_2^{\delta}\nonumber\\
 &{}& +\sup_{N\gtrsim 1}\sum_{N_2\lesssim N^{2\alpha-2}}\,\sum_{L_{max}\sim
N^{2\alpha},L_{med}\ge N^2N_2
}\frac{N^{2\delta}N_2^{\delta}}{L_{max}^{\frac
12-\delta}L_{med}^{\delta}}\nonumber\\
&\lesssim&\sup_{N\gtrsim
1}N^{-\alpha+4\alpha\delta}+\sup_{N\gtrsim 1}\sum_{N_2\lesssim
N^{2\alpha-2}}N^{-\alpha+2\delta(1+\alpha)}N_2^{\delta}\nonumber\\
&\lesssim&\sup_{N\gtrsim 1}N^{-\alpha+4\alpha\delta} \lesssim
1,\label{3.4b}\end{eqnarray} for $\delta>0$  small. By symmetry,
the same estimate  holds when   $N_{min}=N_3$. We complete
 the proof of $(\ref{6})\lesssim 1$.\end{proof}

\begin{theorem}\label{T3.1}Given $s\in (-\min\{\frac
{3+2\alpha}4,1\},-\frac 12)$,  there exists $\mu>0$, $\delta>0$
such that for any  $u, v\in {X}_{\alpha }^{\frac{1}{2},s}$ with
compact support in $[-T, T]$,
\begin{equation} \left\|\partial_x(uv)\right\|_{ {X}_{\alpha
}^{-\frac{1}{2}+\delta,s}}\lesssim T^{\mu}\left\|u\right\|_{
{X}_{\alpha }^{\frac{1}{2},s}}\left\|v\right\|_{ {X}_{\alpha
}^{\frac{1}{2},s}}.\label{3.6}\end{equation}\end{theorem}
\begin{proof} By duality, (\ref{3.6}) is equivalent
to,  for all $w\in  {X}_{\alpha }^{\frac{1}{2}-\delta,s}$,
\begin{equation}\left |<\partial_x(uv), w>\right |\lesssim
T^{\mu}\left\|u\right\|_{ {X}_{\alpha
}^{\frac{1}{2},s}}\left\|v\right\|_{ {X}_{\alpha
}^{\frac{1}{2},s}}\left\|w\right\|_{ {X}_{\alpha
}^{\frac{1}{2},s-\delta}}.\label{3.7}\end{equation} Then the
theorem follows from Lemma 4 in \cite{Monlinet}, (\ref{3.7}) and
 Lemma \ref{L3.2}.\end{proof}

The following theorem is a direct consequence of Theorem
\ref{T3.1} together with the triangle inequality
$$<\xi>^s\le
<\xi>^{s_c}<\xi_1>^{s-s_c}+<\xi>^{s_c}<\xi-\xi_1>^{s-s_c},\,
\forall s\ge s_c.$$

\begin{theorem}\label{T3.2}Given $s_c\in (
-\min\{\frac {3+2\alpha}4,1\},-\frac 12)$,  there  exists $\mu>0$,
$\delta>0$ such that for any $s\ge s_c$ and for any couple
$(u,v)\in {X}_{\alpha }^{\frac{1}{2},s}$ with compact support in
$[-T, T]$,
\begin{equation} \left\|\partial_x(uv)\right\|_{ {X}_{\alpha
}^{-\frac{1}{2}+\delta,s}}\lesssim T^{\mu}\left
(\left\|u\right\|_{ {X}_{\alpha
}^{\frac{1}{2},s_c}}\left\|v\right\|_{ {X}_{\alpha
}^{\frac{1}{2},s}}+\left\|u\right\|_{ {X}_{\alpha
}^{\frac{1}{2},s}}\left\|v\right\|_{ {X}_{\alpha
}^{\frac{1}{2},s_c}}\right
).\label{3.6a}\end{equation}\end{theorem}

{\bf The proof of Theorem \ref{T1.1}.}\quad The proof is similar
to that of Theorem 1 in \cite{Monlinet}, we omit it.

\section{Ill-posedness results}

 In this section we give some  ill-posedness results.
\begin{theorem}\label{T5.1}Let $\frac 12\le\alpha\le 1$, $s<-1$ and  $T>0$. Then there does not exist a space
$Y_{T}$ continuously embedded in $C([0,T],{H}^{s}(\R))$ such that
\begin{equation}\parallel W(t)\varphi
\parallel _{Y_{T}}\lesssim \parallel \varphi \parallel _{H^s},\quad \forall \varphi \in {H}^{s}(\R), \label{22}\end{equation}
\begin{equation}\parallel \int _{0}^{t}W(t-t')\partial
_{x}[u^{2}(t')]dt'\parallel _{Y_{T}}\lesssim \parallel u\parallel
_{Y_{T}}^{2}, \quad \forall u\in Y_{T}.
\label{23}\end{equation}\end{theorem}

\begin{proof} Suppose that there exists a space $Y_{T}$ such that (\ref{22}) and
(\ref{23}) hold. For any $t\in [0,T]$, taking $u=W(t)\varphi $ and
since $Y_{T}$ is continuously embedded in $C([0,T],{H}^{s}(\R))$,
we get
\begin{equation}\parallel \int _{0}^{t}W(t-t')\partial
_{x}[(W(t')\varphi )^{2}]dt'\parallel _{{H}^{s}}\lesssim
\parallel \int _{0}^{t}W(t-t')\partial _{x}[(W(t')\varphi
)^{2}]dt'\parallel _{Y_{T}}\lesssim \parallel \varphi \parallel
_{{H}^{s}}^{2}. \label{24}\end{equation} We show now that
(\ref{24}) fails by choosing an appropriate sequence $\{\varphi
_{N}\}$. Let $\{\varphi _{N}\}$ be the real-valued function
defined through its Fourier transform by \[\hat{\varphi
}_{N}=N^{-s}[\chi_{I_{N}}(\xi )+\chi_{-I_{N}}(\xi )],\] where
$I_{N}=[N,N+2]$, so $\varphi _{N}\in {\cal S}$. Note that
$\parallel \varphi \parallel _{{H}^{s}}\sim 1$, setting
\begin{eqnarray*}u_{1,N}(t,x)=W(t)\varphi_N ,\qquad
u_{2,N}(t,x)=\int _{0}^{t}W(t-t')\partial _{x}[(W(t')\varphi
)^{2}]dt'.\end{eqnarray*} and taking $x$-Fourier transform , we
will get
\[\mathcal{F}_{x}(u_{2,N}(t,\cdot ))(\xi )=\int
_{0}^{t}e^{-(t-t')\mid \xi \mid ^{2\alpha }}e^{i(t-t')\xi
^{3}}(i\xi )
[\mathcal{F}_{x}(u_{1,N}(t'))*\mathcal{F}_{x}(u_{1,N}(t'))](\xi
)dt',\] where
\begin{eqnarray*}&{}&[\mathcal{F}_{x}(u_{1,N}(t'))*\mathcal{F}_{x}(u_{1,N}(t'))](\xi
)= [\mathcal{F}_{x}(W(t')\varphi
_{N})*\mathcal{F}_{x}(W(t')\varphi _{N})](\xi
)\\
&=&\int_{\R} \hat{\varphi }_{N}(\xi _{1})\hat{\varphi }_{N}(\xi
-\xi _{1})e^{-(\mid \xi _{1}\mid ^{2\alpha } +\mid \xi -\xi
_{1}\mid ^{2\alpha })t'}e^{i(\xi _{1}^{3}+(\xi -\xi
_{1})^{3})t'}d\xi _{1}.\end{eqnarray*} Hence
\begin{eqnarray*}&{}&\mathcal{F}_{x}(u_{2,N}(t,\cdot ))(\xi
)=e^{-t\mid \xi \mid ^{2\alpha }}e^{it\xi ^{3}}(i\xi )\int_{\R}
\hat{\varphi }_{N}(\xi _{1})\hat{\varphi }_{N}(\xi -\xi
_{1})\\
&{}&\qquad \times\frac{e^{-(\mid \xi _{1}\mid ^{2\alpha } +\mid
\xi -\xi _{1}\mid ^{2\alpha }-\mid \xi \mid ^{2\alpha })t}e^{i(\xi
_{1}^{3}+(\xi -\xi _{1})^{3}-\xi ^{3})t}-1}{-(\mid \xi _{1}\mid
^{2\alpha } +\mid \xi -\xi _{1}\mid ^{2\alpha }-\mid \xi \mid
^{2\alpha })+i(\xi _{1}^{3}+(\xi -\xi _{1})^{3}-\xi ^{3})} d\xi
_{1},\end{eqnarray*}
\begin{eqnarray*}&{}&\parallel u_{2,N}(t)\parallel
_{{H}^{s}}^{2}\geq \int _{-\frac 1{2}}^{\frac 1{2}}< \xi >^{2s}
\mid \mathcal{F}
_{x}(u_{2,N}(t,\cdot ))(\xi )\mid ^{2}d\xi \\
&=&N^{-4s}\int _{-\frac{1}{2}}^{\frac{1 }{2}}\left | \int _{K_{\xi
}}\frac{e^{-(\mid \xi _{1}\mid ^{2\alpha } +\mid \xi -\xi _{1}\mid
^{2\alpha })t}e^{i(\xi _{1}^{3}+(\xi -\xi _{1})^{3}-\xi
^{3})t}-e^{-\mid \xi \mid ^{2\alpha }t}}{-(\mid \xi _{1}\mid
^{2\alpha } +\mid \xi -\xi _{1}\mid ^{2\alpha }-\mid \xi \mid
^{2\alpha })+i(\xi _{1}^{3}+(\xi -\xi _{1})^{3}-\xi ^{3})}d\xi
_{1}\right |^{2}\\
&{}&\qquad\qquad\times <\xi>^{2s} \mid \xi \mid ^{2} d\xi,
\end{eqnarray*}
where $$K_{\xi }=\{\xi _{1}\mid \xi -\xi _{1}\in I_{N},\xi _{1}\in
-I_{N}\}\cup \{\xi _{1}\mid \xi -\xi _{1}\in -I_{N},\xi _{1}\in
I_{N}\}.$$ Note that for any $\xi \in [-\frac{1 }{2},\frac{1
}{2}]$. One has $mes(K_{\xi })\gtrsim 1 $ and
$$\mid \xi _{1}\mid
^{2\alpha } +\mid \xi -\xi _{1}\mid ^{2\alpha }-\mid \xi \mid
^{2\alpha }\sim N^{2\alpha },\qquad \xi _{1}^{3}+(\xi -\xi
_{1})^{3}-\xi ^{3}=3\xi\xi_1(\xi-\xi_1)\sim N^{2}. $$
We have
\[ e^{-\mid \xi \mid ^{2\alpha }t}- Re(e^{-(\mid \xi _{1}\mid ^{2\alpha } +\mid \xi -\xi _{1}\mid
^{2\alpha })t}e^{i(\xi _{1}^{3}+(\xi -\xi _{1})^{3}-\xi ^{3})t})
\geq e^{-(1 /2)^{2\alpha }t}-e^{-2(N+2)^{2\alpha }t},\] which
leads to
\[\left | \int _{K_{\xi }}\frac{e^{-(\mid \xi _{1}\mid ^{2\alpha }
+\mid \xi -\xi _{1}\mid ^{2\alpha })t}e^{i(\xi _{1}^{3}+(\xi -\xi
_{1})^{3}-\xi ^{3})t}-e^{-\mid \xi \mid ^{2\alpha }t}}{-(\mid \xi
_{1}\mid ^{2\alpha } +\mid \xi -\xi _{1}\mid ^{2\alpha }-\mid \xi
\mid ^{2\alpha })+i(\xi _{1}^{3}+(\xi -\xi _{1})^{3}-\xi
^{3})}d\xi _{1}\right | \geq  \frac{e^{-(1 /2)^{2\alpha
}t}-e^{-2(N+2)^{2\alpha }t}}{N^{2\alpha }+N^{2}}.\] Thus
\begin{equation}\parallel u_{2,N}(t)\parallel _{{H}^{s}}^{2}\geq
N^{-4s}\left (\frac{e^{-(1 /2)^{2\alpha }t}-e^{-2(N+2)^{2\alpha
}t}}{N^{2\alpha }+N^2}\right )^2\geq N^{-4s-4}\left (e^{-(1
/2)^{2\alpha }t}-e^{-2(N+2)^{2\alpha }t}\right )^2
.\label{25}\end{equation}(\ref{25}) contradicts (\ref{24}) when
$N$ is large enough.\end{proof}
\begin{theorem}\label{T4.2} Let $\frac 12\le \alpha\le 1$ and $s<-1$.
Then there does not exists any $T$ such that (\ref{1}) admits a
unique local solution defined on the interval $[0, T]$ and such
that the flow-map $$\varphi \mapsto u(t),\, t\in [0, T]$$ is $C^2$
differentiable at zero from $H^s(\R)$ to $C([0, T];
H^s(\R))$.\end{theorem} \begin{proof} Let $u$ be a solution of
(\ref{1}). Then we have
\[u(t,x,\varphi )=W(t)\phi -\frac{1}{2}\int _{0}^{t}W(t-t')\partial _{x}(u(t', x,\phi )^{2})dt'.\]
Assume now that the flow-map is $C^{2}$. Since $u(t,x,0)\equiv 0$,
we have \[u_{1}(t,x):=\frac{\partial u}{\partial \phi
}(t,x,0)[h]=W(t)h,\]
\begin{eqnarray*}u_{2}(t,x):&=&\frac{\partial ^{2}u}{\partial ^{2}\phi }(t,x,0)[h,h]=\int _{0}^{t}W(t-t')\partial _{x}(u_{1}(t',x))^{2}dt'\\
&=&\int _{0}^{t}W(t-t')\partial
_{x}(W(t')h)^{2}dt'.\end{eqnarray*} Since the flow-map is $C^{2}$
one must have \[\parallel u_{2}(t)\parallel _{{H}^{s}}\leq
\parallel h\parallel _{{H}^{s}}^{2},\quad \forall h\in {H}^{s}(\R).\]
But this is exactly the estimate which has been shown to fail in
the proof of Theorem \ref{T5.1}.\end{proof}

\end{document}